\documentclass[reqno,11pt]{amsart}
\usepackage{amsmath,amssymb,amsfonts,amsthm, amscd,indentfirst}
\usepackage{amsmath,latexsym,amssymb,amsmath,
	amscd,amsthm,amsxtra}
\usepackage{color}
\usepackage{pdfpages}
\usepackage{graphics}
\usepackage{enumitem}

\usepackage[usenames,dvipsnames]{xcolor}

\usepackage{varioref}
\usepackage{hyperref}
\usepackage[noabbrev,capitalise,nameinlink]{cleveref}

\newcommand*\mathinhead[2]{\texorpdfstring{$#1$}{#2}}
\newcommand*\mathinheadbold[2]{\texorpdfstring{$\boldsymbol{#1}$}{#2}}

\hypersetup{colorlinks={true},linkcolor={blue},citecolor=blue}
\usepackage[english]{babel}
\usepackage{csquotes}

\usepackage[backend=biber,
style=alphabetic,datamodel=ext-eprint]{biblatex}
\DeclareSourcemap{
	\maps[datatype=bibtex]{
		\map{
			\step[fieldsource=pmid, fieldtarget=pubmed]
		}
	}
}
\makeatletter
\DeclareFieldFormat{arxiv}{%
	arXiv\addcolon\space
	\ifhyperref
	{\href{http://arxiv.org/\abx@arxivpath/#1}{%
			\nolinkurl{#1}%
			\iffieldundef{arxivclass}
			{}
			{\addspace\texttt{\mkbibbrackets{\thefield{arxivclass}}}}}}
	{\nolinkurl{#1}
		\iffieldundef{arxivclass}
		{}
		{\addspace\texttt{\mkbibbrackets{\thefield{arxivclass}}}}}}
\makeatother
\DeclareFieldFormat{pmcid}{%
	PMCID\addcolon\space
	\ifhyperref
	{\href{http://www.ncbi.nlm.nih.gov/pmc/articles/#1}{\nolinkurl{#1}}}
	{\nolinkurl{#1}}}
\DeclareFieldFormat{mr}{%
	MR\addcolon\space
	\ifhyperref
	{\href{http://www.ams.org/mathscinet-getitem?mr=MR#1}{\nolinkurl{#1}}}
	{\nolinkurl{#1}}}
\DeclareFieldFormat{zbl}{%
	Zbl\addcolon\space
	\ifhyperref
	{\href{http://zbmath.org/?q=an:#1}{\nolinkurl{#1}}}
	{\nolinkurl{#1}}}

\renewbibmacro*{eprint}{%
	\printfield{arxiv}%
	\newunit\newblock
	\printfield{jstor}%
	\newunit\newblock
	\printfield{mr}%
	\newunit\newblock
	\printfield{zbl}%
	\newunit\newblock
	\printfield{hdl}%
	\newunit\newblock
	\printfield{pubmed}%
	\newunit\newblock
	\printfield{pmcid}%
	\newunit\newblock
	\printfield{googlebooks}%
	\newunit\newblock
	\iffieldundef{eprinttype}
	{\printfield{eprint}}
	{\printfield[eprint:\strfield{eprinttype}]{eprint}}}

\usepackage{graphicx}
\usepackage{caption}

\usepackage{epsfig,here}
\usepackage{subfigure,here}

\newtheorem{theorem}{Theorem}[section]
\newtheorem{lemma}[theorem]{Lemma}
\newtheorem{proposition}[theorem]{Proposition}
\newtheorem{definition}[theorem]{Definition}
\newtheorem{remark}[theorem]{Remark}
\def\p{\partial}

\def\p{\partial}

\def\<{\langle}
\def\>{\rangle}

\def\ep{\epsilon}
\def\om{\omega}
\def\Om{\Omega}

\def\p{\partial}

\newcommand{\mfR}{\mathbf{R}}

\newcommand{\mfS}{\mathbf{S}}

\newcommand{\mcB}{\mathcal{B}}
\newcommand{\mcC}{\mathcal{C}}

\newcommand{\mcH}{\mathcal{H}}

\newcommand{\mcU}{\mathcal{U}}
\newcommand{\ra}{\rightarrow}

\newcommand{\rd}{{\rm d}}

\numberwithin{equation} {section}

\begin{document}
	
	\title[\mathinhead{C^{1,1}}{C11}-rectifiability on sphere
	]{\mathinheadbold{C^{1,1}}{C11}-rectifiability and Heintze-Karcher inequality on $\mfS^{n+1}$}
	
	\author{Xuwen Zhang}
	\address{School of Mathematical Sciences\\
		Xiamen University\\
		361005, Xiamen, P.R. China}
	\email{xuwenzhang@stu.xmu.edu.cn}
	
	\begin{abstract}
		In this paper, by isometrically embedding $(\mfS^{n+1},g_{\mfS^{n+1}})$ into $\mfR^{n+2}$, and using nonlinear analysis on the codimension-2 graphs, we will show that the level-sets of the distance function from the boundary of any open set in sphere, are $C^{1,1}$-rectifiable. As a by-product, we establish a Heintze-Karcher inequality on sphere.
		
	\end{abstract}
	
	\maketitle
	
	
	
	
	
	\section{Introduction}
The isoperimetric theorem, a fundamental but important topic in the calculus of variations, has attached well attention of many mathematicians. From the perspective of the modern calculus of variations, sets of finite perimeter are believed to be the natural competition class in which the isoperimetric theorem shall be formulated.

Starting from De Giorgi \cite{DeGiorgi54,DeGiorgi55}, who managed to show by using Steiner symmetrization and the compactness theorem of sets of finite perimeter that Euclidean balls are the only isoperimetric sets (global minimizers) among sets of finite perimeter, mathematicians have been working in various context of minimizers to study the isoperimetric problem for decades. Such problem is already found to be very subtle in the context of local minimizers, due to the lack of regularity in the higher dimensional situation (and hence the classical moving plane method fails to be applicable), see \cite{SZ18} for an example of local volume-constrained perimeter minimizer admitting singularities. Despite these obstacles, very recently, Delgadino-Maggi \cite{DM19} solved the very important open problem: the characterization of critical points of the Euclidean isoperimetric problem among sets of finite perimeter. In the weakest assumption (critical points of Euclidean isoperimetric problem), they obtained the following (see also \cite{RKS20} for the anisotropic version, which is solved by using a completely different method with \cite{DM19}).
\begin{theorem}[{\cite[Theorem 1]{DM19}}]\label{Thm-DM19-1}
	Among sets of ﬁnite perimeter and ﬁnite volume, ﬁnite unions of balls with equal radii are the unique critical points of the Euclidean isoperimetric problem.
\end{theorem}
Their advanced techniques to approach this problem are two-fold: 1. By using subtle nonlinear analysis and geometric measure theory, they established the following $C^{1,1}$-rectifiability result of the level sets of the distance function from the boundary of any bounded sets of finite perimeter in $\mfR^{n+1}$.
\begin{proposition}[{\cite[Lemma 7]{DM19}}]\label{Prop-DM7}
   If $\Omega$ is an open set with finite perimeter and finite volume in $\mfR^{n+1}$, then $\Omega_s=\{y\in\Omega:u(y)>s\}$ is an open set of finite perimeter with $\mcH^n(\p\Omega_s\setminus\Gamma_s^+)=0$ for a.e. $s>0$, here $u(y)={\rm dist}(y,\p\Omega)$ is the distance function from $\p\Omega$, defined for every $y\in\Omega$, $\Gamma_s^+=\bigcup_{t>0}\Gamma_s^t$, where $\Gamma_s^t$ is defined as
   \begin{align}
       \Gamma_s^t=\left\{y\in\p\Omega_s:y=(1-\frac{s}{t})x+\frac{s}{t}z\quad\text{for some }z\in\p\Omega_t,x\in\p\Omega\right\}.
   \end{align}
   Moreover, for every $s>0$, $\Gamma_s^+$ can be covered by countably many graphs of $C^{1,1}$-functions from $\mfR^n$ to $\mfR^{n+1}$.
\end{proposition}
The $C^{1,1}$-rectifiability of $\Gamma_s^t$ allows them to define the principal curvatures a.e. on $\Gamma_s^+$, which are bounded from above and by below due to the definition of $\Gamma_s^t$. As a by-product, they followed the proof of \cite{Bre13} and derived a Heintze-Karcher inequality for sets of finite perimeter (\cite[Theorem 8]{DM19}). We note that although \cref{Prop-DM7} is stated in the context of sets of finite perimeter,
the statement remains valid when $\Omega$ is only a bounded open set, we refer to the details of the proof in \cite[Step1,2,3]{DM19}.

2. They exploited the Sch\"attzle's strong maximum principle for the codimension-1 integer rectifiable varifolds \cite[Theorem 6.2]{Schatzle04} and showed that the flow method used by Montiel-Ros in \cite[Theorem 3]{MR91} to prove the Heintze-Karcher inequality for $C^2$-closed hypersurfaces can be modified to apply in the context of sets of finite perimeter, thus proved their main theorem \cref{Thm-DM19-1}.
	
It is worth mentioning that aforementioned works on the isoperimetric problems are contextualized in the Euclidean space. In view of \cite{Rei80,Ros87,MR91}, a natural question arises: is there any characterization of geodesic balls as the only critical points in the isoperimetric problem that is brought up in space forms?

Motivated by this natural question and the celebrated work \cite{DM19}, in this note, we follow the subtle nonlinear analysis carried out by Delgadino-Maggi and prove a $C^{1,1}$-rectifiability result in $(\mfS^{n+1},g_{\mfS^{n+1}})$.
On the other hand, we manage to prove a Heintze-Karcher type inequality for any open set that is contained in a hemisphere, enlightened by Brendle's monotonicity approach \cite{Bre13}. To introduce our main results, let us first fix some notations.

Let $(\mfS^{n+1},g_{\mfS^{n+1}})$ be the space form with sectional curvature which is identically 1, for simplicity, we abbreviate it by $\mfS^{n+1}$ in the rest of this paper. Let ${\rm dist}_g$ denote the distance function of $\mfS^{n+1}$.

For any nonempty open set $\Omega\subset\mfS^{n+1}$, we define $u(y)={\rm dist}_g(y,\p\Omega)$ to be the distance function with respect to $\p\Omega$ on $\mfS^{n+1}$.
For $s>0$, we define the super level-set and the level-set of $u$ in $\Omega$ by
\begin{align}\label{defn-Omegas}
    \Omega_s=\{y\in\Omega:u(y)>s\},\quad\p\Omega_s=\{y\in\Omega:u(y)=s\}.
\end{align}
As a parallel version of the Euclidean one in \cite{DM19}, we introduce the following definitions.
\begin{definition}[$\Gamma_s^t$ and $\Gamma_s^+$]\label{defn-Gamma_s^t}
\normalfont
For any nonempty open set $\Omega\subset\mfS^{n+1}$, and for every $0<s<t<\pi$,
define $\Gamma_s^t$ to be the set of points of $\Om$ such that: for $y\in\Om$, there holds
\begin{enumerate}
    \item ${\rm dist}_g(y,\p\Omega)=s$,
    \item there exists a unit-speed geodesic $\gamma_y:[s-t,s]\ra\mfS^{n+1}$, with $\gamma(0)=y$ and ${\rm dist}(\gamma(r),\p\Om)=s-r$ for every $r\in[s-t,s]$.
\end{enumerate}
Moreover, for every $0<s<\pi$, let
    \begin{align*}
\Gamma_s^+:=\bigcup_{s<t<\pi}\Gamma_s^t.
	\end{align*}
\end{definition}
\subsection{Main results}

The main purpose of this paper is to establish the following $C^{1,1}$-rectifiability result, which extends \cref{Prop-DM7} to $\mfS^{n+1}$.
\begin{theorem}\label{Thm-C11Rec}
For any nonempty open set $\Omega\subset\mfS^{n+1}$ and for every $0<s<t<\pi$,
there exists a countable collection $\{\mcU_j\}_{j\geq1}$ of compact subsets of $\Gamma_s^t$ such that $\mcH^n(\Gamma_s^t\setminus\bigcup_{j=1}^\infty \mcU_j)=0$, with each $\mcU_j$ contained in a $C^{1,1}$-hypersurface.
Moreover, denoting by $N$ the gradient of the distance function with respect to $u(\cdot)={\rm dist}_g(\cdot,\p\Omega)$, then $N\mid_{\mcU_j}$ is Lipschitz for every $j\geq1$.
\end{theorem}
With the $C^{1,1}$-rectifiability in force, the principal curvatures $(\kappa_s^t)_i$ of $\Gamma_s^t$, the viscosity boundary $\p^\nu\Omega$ of $\Omega$, and the viscosity mean curvature $H_\Omega^\nu$ of $\Omega$ are thus naturally defined in \cref{Prop3-summary}, \cref{Defn-2ndfundamentalform} and \cref{Defn-ViscosityMeanConvex}, see also \cite[Lemma 7]{DM19} for the Euclidean version.

Consequently, following Brendle's monotonicity approach \cite[Section 3]{Bre13}, we prove the following Heintze-Karcher type inequality on $\mfS^{n+1}$.

\begin{theorem}[Heintze-Karcher inequality on Sphere]\label{Thm-HK}
	If $\Om\subset\mfS^{n+1}$ is an open set lying completely in a hemisphere, which is mean convex in the viscosity sense as in \cref{Defn-ViscosityMeanConvex}, then for every $0<s<\frac{\pi}{2}$,
	\begin{align}\label{HK-INEQ}
		\int_{s}^{\frac{\pi}{2}}\cos{r}\mcH^n(\p\Om_r)\rd r\leq\frac{n}{n+1}\int_{\Gamma_s^+}\frac{\cos{s}}{H_{\Om_s}}\rd\mcH^n.
	\end{align}
	Moreover, as $s\ra0^+$,
	the limit of the right-hand side of \eqref{HK-INEQ} always exists in $(0,\infty]$.
\end{theorem}

Our strategy for proving the main rectifiability theorem follows largely from \cite{DM19} and is as follows:
by isometrically embedding $\mfS^{n+1}$ into $\mfR^{n+2}$, our goal becomes: to show that the aforementioned $\Gamma_s^t$ is a $n$-rectifiable set in $\mfR^{n+2}$. 
Using the Hinge version of Topogonov theorem, we obtain an estimation of $\vert N(y)\cdot (y'-y)\vert$, where $y',y\in\Gamma_s^t$ and $N(y)$ is the derivative of the distance function at $y$, which will be proved to be well-defined everywhere on $\Gamma_s^t$ in \cref{Prop-Gammast},
both $y,y'$, and $N(y)$ are considered as vectors in $\mfR^{n+2}$, $`\cdot'$ denotes the standard Euclidean inner product.
By virtue of this basic estimation, we can use the $C^1$-Whitney extension theorem and the $C^1$-implicit function theorem to show that $\Gamma_s^t$ is $C^1$-rectifiable.
To prove the $C^{1,1}$-rectifiability, in the Euclidean case, Delgadino-Maggi's proof of $C^{1,1}$-rectifiability is built on the fact that $\Gamma_s^t$ can be written as a codimension-1, $C^1$-graph in $\mfR^{n+1}$.
In our situation, the main obstacle to prove the $C^{1,1}$-rectifiability is that, the analysis of codimension-2 graph in $\mfR^{n+2}$ seems more subtle. Regarding this, our approach can be viewed as a codimension-2 counterpart of the one presented in \cite[Theorem1, step1]{DM19}.

In view of the classical rigidity result of $C^2$, CMC hypersurfaces in $\mfS^{n+1}$ \cite[part C)]{MR91},
and the Alexandrov theorem for critical points of isoperimeteric problem among sets of finite perimeter \cref{Thm-DM19-1} (see in particular \cite[Theorem1, step4]{DM19}),
it is interesting to see whether one can establish an Alexandrov theorem on $\mfS^{n+1}$ among sets of finite perimeter,
we hope that our rectifiability result can serve as a fundamental step for solving this interesting open problem.
On the other hand, we believe our codimension-2 analysis can be used in a wider range of problems that deal with graphs on $\mfS^{n+1}$.

{\color{black}As pointed out to us by a referee, our rectifiability theorem can be deduced quite
easily from the corollary of \cite[Theorem 4.12]{MS19}, which asserts that \textit{if $S$ is a closed subset of $\mfR^{n+1}$ and $S^\ast$ is the set of $x\in S$ such that there exists an open ball $B_x$ with $B_x\cap S=\emptyset$ and $x\in\p B_x$, then $S^\ast$ can be $\mcH^n$-almost covered by the union of countable collection of $C^2$-hypersurfaces.}
This theorem readily implies that (using local charts) all the level-sets of the distance function from an arbitrary closed set in a complete $(n+1)$-dimensional Riemannian manifold can be $\mcH^n$-almost covered by the union of embedded $C^2$-hypersurfaces.
This fact, together with the applications of a few standard tools in Geometric Measure Theory, yields \cref{Thm-C11Rec} even when the standard sphere $(\mfS^{n+1},g_{\mfS^{n+1}})$ is replaced by an arbitrary complete Riemannian manifold and for all $0<s<t<\infty$.
In this regard, our proof of \cref{Thm-C11Rec} serves as an alternative approach by using the nonlinear analysis arguments.}

\subsection{Organization of the paper}

In \cref{Sec-2} we collect some background material from geometric measure theory. In \cref{Sec3} we study the fine properties of $\Gamma_s^t$. In \cref{Sec4} we prove the main rectifiability result \cref{Thm-C11Rec}. In \cref{Sec-5}, we define the viscosity mean curvature and boundary of any open set in $\mfS^{n+1}$, and we establish a Heintze-Karcher type inequality \cref{Thm-HK}.

\subsection{Acknowledgements}

I would like to express my deep gratitude to my advisor Chao Xia for many helpful discussions, constant encouragement, and for bringing \cite{DM19} to my attention. I would wish to thank Wenshuai Jiang, Liangjun Weng, and Xiaohan Jia for several helpful discussions. I would also like to thank Mario Santilli for the enlightening discussions and for informing the related works on the $C^{1,1}$-rectifiability theorem and the Heintze-Karcher inequality and to thank the anonymous referee for reading the manuscript and providing useful comments which help improve the exposition of this paper.

\section{Rectifiable sets}\label{Sec-2}

The main purpose of this note is to establish the rectifiability result, here we list some fundamental concepts and tools that are needed in the sequel, and we refer to \cite{Sim83, DeL08,Mag12} for more details. We must point out that, by virtue of the embedding $\mfS^{n+1}\hookrightarrow\mfR^{n+2}$, in most of this paper, we will be working in $\mfR^{n+2}$, and we use $\mcH^k$ to denote the $k$-dimensional Hausdorff measure in $\mfR^{n+2}$.
\subsection{Rectifiable set}

\begin{definition}[$n$-rectifiable set, {\cite[Section 2.1]{DM19}}]
\normalfont
    A Borel set $N\subset\mfR^{n+2}$ is a \textit{locally $\mcH^n$-rectifiable set} if $N$ can be covered, up to a $\mcH^n$-negligible set, by countably many Lipschitz images of $\mfR^n$ to $\mfR^{n+2}$, and if $\mcH^n\llcorner N$ is locally finite on $\mfR^{n+2}$.
    $N$ is called \textit{$\mcH^n$-rectifiable} if in addition, $\mcH^n(N)<\infty$.
\end{definition}

\subsection{Area formula and Coarea formula}

\begin{proposition}[{Area formula for $k$-rectifiable sets, \cite[Theorem 11.6]{Mag12},\cite[(12.4)]{Sim83}}]
    For $1\leq k\leq m$, if $A\subset\mfR^n$ is a $\mcH^k$-rectifiable set and $f:\mfR^n\ra\mfR^m$  is a Lipschitz map, then
	\begin{align}\label{Areaformula}
		\int_{\mfR^m}\mcH^0\left(A\cap \left\{ f
		=y\right\} \right)\rd\mcH^k(y)=\int_A J^Af(x)\rd\mcH^k(x),
	\end{align} 
	where $\left\{ f:=y\right\}=\left\{x\in\mfR^n: f(x)=y\right\}$, $J^Af(x)$ is the Jacobian of $f$ with respect to $A$ at $x$(see for example \cite[(12.3)]{Sim83} and \cite[(11.1)]{Mag12}), which exists for $\mcH^k$-a.e. $x\in A$.
\end{proposition}

\begin{proposition}[{Coarea formula for $k$-rectifiable set, \cite[(12.6)]{Sim83}}]
    For $k\geq m$, if $A\subset\mfR^n$ is a $\mcH^k$-rectifiable set and $f:\mfR^n\ra\mfR^m$  is a Lipschitz map, then
	\begin{align}\label{Coareaformula-rec.set}
	\int_{\mfR^m}\mcH^{k-m}(A\cap f^{-1}(y))\rd\mcH^m(y)=\int_A J^Af(x)\rd\mcH^k(x).
	\end{align}
\end{proposition}
The following proposition is also needed in our codimension-2 argument,
{\color{black}which amounts to be a simple modification of \cite[Section 2.1(iv)]{DM19}.
The proof follows exactly from \cite{DM19} and hence is omitted here.
\begin{lemma}\label{DM-2.1(iv)}
    Let $M\subset\mfR^{n+2}$ be a locally $\mcH^n$-rectifiable set and $f:M\ra\mfR^{n+2}$ is a Lipschitz map defined on $M$, then for any Lipschitz functions $F,G:\mfR^{n+2}\ra\mfR^{n+2}$ such that $F=G=f$ on $M$, we have
    \begin{align}
        \nabla^MF=\nabla^MG\quad\mcH^n\text{-a.e. on }M. 
    \end{align}
    In particular, if $\psi:\mfR^n\ra\mfR^{n+2}$ is a Lipschitz map and $E\subset\mfR^n$ is a Borel set,
    then $T_xM=(\nabla\psi)_{\psi^{-1}(x)}[\mfR^n]$ for $\mcH^n$-a.e. $x\in M\cap\psi(E)$, with
    \begin{align}\label{DM-2.5}
        (\nabla^MF)_x[\tau]=\nabla(F\circ\psi)_{\psi^{-1}(x)}[(\nabla\psi)_x^{-1}[\tau]]\quad\forall\tau\in T_xM.
    \end{align}
\end{lemma}
We note that $\nabla^MF(x)$ denotes the tangential differential of $F$ with respect to $M$ at $x$, which exists for $\mcH^n$-a.e. $x\in M$ by virtue of the Rademacher-type theorem \cite[Theroem 11.4]{Mag12}. 
}

To close this section, we list some well-known facts about the space form $\mfS^{n+1}$, which will be needed in our proof. Note that part of them are already mentioned in the introduction.
\subsection{Geometry of \mathinhead{(\mfS^{n+1},g_{\mfS^{n+1}})}{S,g}}

In this paper, we consider the unit sphere $\mfS^{n+1}\subset\mfR^{n+2}$, endowed with the Riemannian metric $g_{\mfS^{n+1}}$ that is induced from the isometrically embedding $\mfS^{n+1}\hookrightarrow(\mfR^{n+2},g_{\rm euc})$, and we have the following basic facts:
\begin{enumerate}
	\item $\mfS^{n+1}$ is a smooth, complete, compact Riemannian manifold without boundary, having constant sectional curvature which is identically 1.
	\item The injective radius of $\mfS^{n+1}$ is $\pi$, i.e., ${\rm inj}(\mfS^{n+1})=\pi$.
	\item The only geodesics on $\mfS^{n+1}$ are great circles.
	\item For $x\neq z\in\mfS^{n+1}$ with ${\rm dist}_g(x,z)<\pi,$ there exists a unique minimizing unit-speed geodesic joining $x$ and $z$.
	\item The geodesic ball of radius $r$ centered at some $x\in\mfS^{n+1}$, denoted by $\mcB_r(x)$, is umbilical in $\mfS^{n+1}$ with constant principal curvatures $\cot r$.
\end{enumerate}


\section{Fine Properties of \mathinhead{\Gamma_s^t}{Gamma-s-t}}\label{Sec3}
In this section, we explore the fine properties of $\Gamma_s^t$.
For any nonempty open set $\Omega\subset\mfS^{n+1}$, we define $\Gamma_s^t$ and $\Gamma_s^+$ as in \cref{defn-Gamma_s^t}. Recall that for any $y\in\Omega$, we define $u(y)={\rm dist}_g(y,\p\Omega)$ to be the distance function from $\p\Omega$. Following \cite{Fed59}, we define the unique point projection mapping on $\mfS^{n+1}$, see \cite[Definition 4.1]{Fed59} for the Euclidean version.
\begin{definition}
    \normalfont
For any nonempty open set $\Omega\subset\mfS^{n+1}$, let ${\rm Unp}(\p\Omega)$ be the set of all those points $y\in \Omega$ for which there exists a unique point of $\p\Omega$ nearest to $y$, and the map
\begin{align}
    \xi:{\rm Unp}(\p\Omega)\ra \p\Omega
\end{align}
associates with $y\in{\rm Unp}(\p\Omega)$ the unique $x\in \p\Omega$ such that $u(y)={\rm dist}_g(x,y)$.
\end{definition}
Our first observation is that $\Gamma_s^t\subset{\rm Unp}(\p\Omega)$.
\begin{lemma}\label{Lem-Gammast1}
    For any nonempty open set $\Omega\subset\mfS^{n+1}$ and for every $0<s<t<\pi$, let
    $\Gamma_s^t$ be as in \cref{defn-Gamma_s^t}.
    Then, for any $y\in\Gamma_s^t$, it has a unique point projection onto $\p\Om$, which reads as $\xi(y)=x$; in other words, $\Gamma_s^t\subset{\rm Unp}(\p\Om)$.
    Similarly, it has a unique point projection onto $\p\Om_t$.
\end{lemma}
\begin{proof}
By definition, for any $y\in \Gamma_s^{t}$,
there exists $x=\gamma_y(s)\in\p\Om$ and $z=\gamma_y(s-t)\in\p\Om_{t}$.

Since $y\in\p\Omega_s$, if there exists $x'\neq x\in\p\Omega$ such that ${\rm dist}_g(x',y)=s$, then by the triangle inequality
we have
	\begin{align*}
		{\rm dist}_g(x',z)< {\rm dist}_g(x',y)+{\rm dist}_g(y,z)
		=s+t-s
		=t,
	\end{align*}
	contradicts to the fact that $z\in\p\Omega_t$. Therefore, we have showed that $x$ is the unique point of $\p\Omega$ nearest to $y$; that is, $\xi(y)=x$.

On the other hand, we note that by using the triangle inequality again, it is easy to see that $z$ is the unique point in $\p\Om_t$ such that ${\rm dist}_g(y,z)=t-s$.
\end{proof}
Now that we have showed that $\Gamma_s^t\subset{\rm Unp}(\p\Omega)$, we can explore the unique point projection mapping $\xi$ on $\Gamma_s^t$. Indeed, we have the following.
\begin{lemma}\label{Lem-u-xi}
	For any nonempty open set $\Om\subset\mfS^{n+1}$, the following statements hold:
	\begin{enumerate}
		\item \label{lem3.1-item1}the function $u(\cdot)={\rm dist}_g(\cdot,\p\Om)$ is a Lipschitz function on $\mfS^{n+1}$ with Lipschitz constant at most 1, i.e., for any $x,y\in \mfS^{n+1}$,
		\begin{align*}
			\vert u(y)-u(x)\vert\leq{\rm dist}_g(x,y).
		\end{align*}
		\item \label{lem3.1-item2}For $0<s<t<\pi$, ${ \xi}$ is continuous on $\Gamma_s^t$.
	\end{enumerate}
\end{lemma}

\begin{proof}
	For any $x,y\in\mfS^{n+1}$. Since $\p\Om\subset\mfS^{n+1}$ is compact, we may let $a\in\p\Om$ such that $u(x)={\rm dist}_g(a,x)$. Without loss of generality, assume that $u(y)\geq u(x)$.
	By the triangle inequality, we find
	\begin{align*}
		\vert u(y)-u(x)\vert=u(y)-u(x)\leq {\rm dist}_g(a,y)-{\rm dist}_g(a,x)\leq{\rm dist}_g(x,y),
	\end{align*}
	which proves {\bf(1)}.
	
	For {\bf(2)}, suppose on the contrary that there exists some $\ep>0$ and a sequence of points $y_1,y_2,y_3,\ldots\in\Gamma_s^t$, converges to $y\in\Gamma_s^t$, such that ${\rm dist}_g(\xi(y_i),\xi(y_j))\geq\ep$ for $i=1,2,\ldots$
	
    By definition, for each $i$, we have, for $i$ large, there holds
	\begin{align}
	    {\rm dist}_g\left(\xi(y_i),y_i\right)=u(y_i)=s.
	\end{align}
	Using the triangle inequality and the fact that $y_i$ converges to $y$, we find
	\begin{align*}
		{\rm dist}_g(\xi(y_i),y)\leq{\rm dist}_g(\xi(y_i),y_i)+{\rm dist}_g(y_i,y)=s+{\rm dist}_g(y_i,y)<s+\ep.
	\end{align*}
	This means, all the points $\{\xi(y_i)\}_i$ are lying in $\p\Om\cap\mcB_{s+\ep}(y)$, which is a bounded subset of the compact set $\p\Om$, and hence by passing to a subsequence, we can assume that $\{\xi(y_i)\}_i$ converges to some point $x\in\p\Om$.
	But then, since $u$ is continuous on $\mfS^{n+1}$, we have
	\begin{align*}
		u(y)=\lim_{i\ra\infty}u(y_i)=\lim_{i\ra\infty}{\rm dist}_g(\xi(y_i),y_i)={\rm dist}_g(x,y),
	\end{align*} 
	which implies that $x=\xi(y)$ since we have proved that $y\in\Gamma_s^t\subset{\rm Unp}(\p\Omega)$ in \cref{Lem-Gammast1}. However, this contradicts to the assumption that
	\begin{align*}
		{\rm dist}_g(x,\xi(y))=\lim_{i\ra\infty}{\rm dist}_g(\xi(y),\xi(y_i))\geq\ep,
	\end{align*}
	and hence completes the proof.
\end{proof}

\begin{remark}
	\normalfont
	When $\Om$ is contained in a Euclidean space, similar results are included in \cite[4.8(1),(4)]{Fed59}
\end{remark}

With the help of \cref{Lem-Gammast1} and \cref{Lem-u-xi}, we collect the fine properties of $\Gamma_s^t$ as follows (see \cite[Theorem1]{DM19} for the Euclidean version).

\begin{proposition}\label{Prop-Gammast}
	For any nonempty open set $\Om\subset\mfS^{n+1}$ and for every $0<s<t<\pi$,
	the following statements hold:
	\begin{enumerate}
		\item \label{prop3.1-item2}For $s<t_1<t_2<\pi$, $\Gamma_s^{t_2}\subset\Gamma_s^{t_1}$. In particular, $\Gamma_s^+=\lim_{t\ra s^+}\Gamma_s^t$.
		\item \label{prop3.1-item3}$\Gamma_s^t$ is a compact set in $\mfS^{n+1}$.
		\item \label{prop3.1-item4}for $y\in\Gamma_s^t$, $\Gamma_s^t$ is bounded by two geodesic balls in $\mfS^{n+1}$, mutually tangent at $y$, i.e.,
		\begin{align*}
			\begin{cases}
				&\mcB_{t-s}(z)\subset\Om_s\subset \mfS^{n+1}\setminus\mcB_{s}(x),\\
				&\{y\}=\p\mcB_{t-s}(z)\cap\p\mcB_s(x).
			\end{cases}
		\end{align*}
		\item \label{prop3.1-item5}the distance function $u$ is differentiable at every $y\in\Gamma_s^t$.
	\end{enumerate}
\end{proposition}

\begin{proof}
	\noindent\textbf{(1)}
	The first part of the statement follows easily from the definition of $\Gamma_s^t$, while by virtue of the inclusion, it is apparent that $\Gamma_s^+=\lim_{t\ra s^+}\Gamma_s^t$.
	
	\noindent{\textbf{(2)}} It suffice to prove that $\Gamma_s^t$ is a closed set in $\mfS^{n+1}$, i.e., if a sequence of points in $\Gamma_s^t$, say $\{y_i\}_{i=1}^\infty$, converges to $y$, then it must be that $y\in\Gamma_s^t$.
	
	By definition of $\Gamma_s^t$, for each $y_i$, there exists corresponding points $x_i\in\p\Om, z_i\in\p\Om_t$. By \cref{Lem-u-xi}, $\xi$ is continuous on $\Gamma_s^t$, and hence we have: $\{x_i\}_{i=1}^\infty$ is a Cauchy sequence\footnote{By Cauchy sequence we mean, for any $\epsilon>0$, there exists some positive integer $N$, such that for any $m,n\geq N$, there holds ${\rm dist}_g(x_m,x_n)<\epsilon$. This shows that $\{x_i\}_{i=1}^\infty$ is a bounded sequence in $\p\Omega$.} in $\p\Om$. Notice that $\p\Om$ is closed, hence $\{x_i\}_{i=1}^\infty$ converges to some $x\in\p\Om$. Similarly, $\{z_i\}_{i=1}^\infty$ converges to some $z\in\p\Om_t$. 
	
	By continuity, we have
	\begin{align*}
		u(y)=\lim_{i\ra\infty}u(y_i)=s,\quad{\rm dist}_g(x,y)=\lim_{i\ra\infty}{\rm dist}_g(y_i,\xi(y_i))=s.
	\end{align*}
	Similarly, we deduce that ${\rm dist}_g(y,\p\Om_t)={\rm dist}_g(y,z)=t$.
	By using the triangle inequality, we find
	\begin{align*}
		t=u(z)\leq{\rm dist}_g(x,z)\leq{\rm dist}_g(x,y)+{\rm dist}_g(y,z)=t,
	\end{align*}
	it is easy to see that the geodesic joining $x$ and $z$ must pass through $y$, and hence $y\in\Gamma_s^t$ by definition. 
	
	\noindent\textbf{(3)} can be deduced from the definition of $\Gamma_s^t$;
	\textbf{(4)} is a direct consequence of the fact that $y\in\Gamma_s^t\subset{\rm Unp}(\p\Omega)$.
\end{proof}
\begin{remark}
    \normalfont
    The lemmas and propositions in this section can be extended easily to the case when $(\mfS^{n+1},g_{\mfS^{n+1}})$ is replaced by any complete Riemannian manifold.
\end{remark}
\section{\mathinhead{C^{1,1}}{C11}-rectifiability of \mathinhead{\Gamma_s^t}{Gamma-s-t}}\label{Sec4}
In this section, we finish the proof of \cref{Thm-C11Rec}.
As mentioned in the introduction, our proof is based on the isometrically embedding $\mfS^{n+1}\hookrightarrow\mfR^{n+2}$.
We point out that, in the rest of the paper, we will be working in $\mfR^{n+2}$. For every $0<s<t<\pi$, and for any nonempty open set $\Om\subset\mfS^{n+1}$, we define $\Gamma_s^t,\Gamma_s^+,\p\Omega_s$ as in \cref{defn-Gamma_s^t}, \eqref{defn-Omegas}, respectively. By virtue of \cref{Prop-Gammast}(4), $u$ is differentiable at $y$, and we denote by $N(y)$ its gradient, which belongs to $T_y\mfS^{n+1}$. In all follows, thanks to the embedding, $N(y)$ will be considered as a vector in $\mfR^{n+2}$.

The following well-known fact motivates our estimation.
\begin{lemma}\label{Lem-xyz-N}
	For any $y\in\Gamma_s^t$, let $x=\gamma_y(s)\in\p\Omega$ and $z=\gamma_y(s-t)\in\p\Omega_t$ be the corresponding points that $y$ admits.
	Then there holds
	\begin{align}
	    &N(y)=-\frac{x+\cos sy}{\sin s}=-\frac{1}{\sin s}x+\frac{\cos s}{\sin s}y,\label{eq-N-xy}\\
	    &z=\frac{y+\tan{(t-s)} N(y)}{\frac{1}{\cos{(t-s)}}}=\cos{(t-s)} y+\sin{(t-s)}N(y)\label{eq-N-yz}.
	\end{align}
\end{lemma}
\begin{proof}
	These are well-known facts and one can check by a direct computation. Notice for example that $y-\tan sN(y)=\frac{1}{\cos s}x$. See \cref{Fig-1} for an illustration.
\end{proof}
\begin{figure}[h]
	\centering
	\includegraphics[height=6cm,width=10cm]{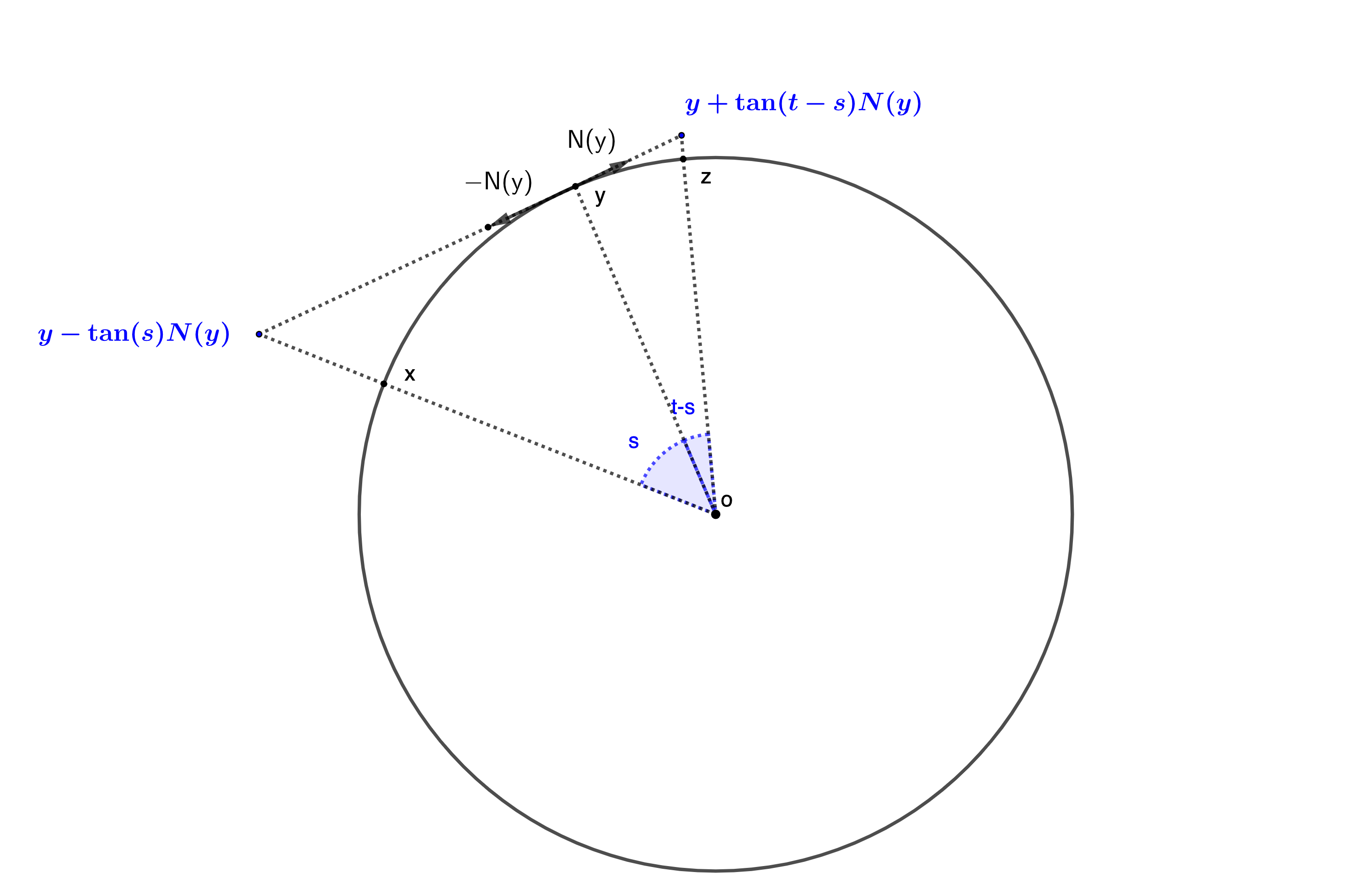}
	\caption{Relation of $x,y,z$ and $N(y)$}
	\label{Fig-1}
\end{figure}

\begin{proof}[Proof of \cref{Thm-C11Rec}]
Throughout the proof, $\vert\cdot\vert$ will denote the Euclidean norm in $\mfR^{n+2}$,
	$\nabla$ will denote the gradient in Euclidean space and $``\cdot"$ will denote the Euclidean inner product in $\mfR^{n+2}$. To make a distinction, we use $\left<\cdot,\cdot\right>$ to denote the Euclidean inner product in $\mfR^n$.
	For any two points $x\neq y\in\mfS^{n+1}$ such that ${\rm dist}_g(x,y)<\pi$, we use $\underline{xy}$ to denote the unique image of the minimizing geodesic on $\mfS^{n+1}$ joining $x$ and $y$.
	
\noindent{\bf Step1. }$C^1$-rectifiability of $\Gamma_s^t$.

	First we estimate $\left\vert N(y)\cdot(y'-y)\right\vert$ in the Euclidean space $\mfR^{n+2}$ for any $y,y'\in\Gamma_s^t$ satisfying ${\rm dist}_g(y',y)<\pi$. A key observation is that, if we denote by $\nu_{\mfS^{n+1}}$ the outwards pointing unit normal of $\mfS^{n+1}$ in $\mfR^{n+2}$, then trivially, $y=\nu_{\mfS^{n+1}}(y)$, and hence for any $y\in\Gamma_s^t\subset\mfS^{n+1}$, we have $N(y)\cdot y=0$.
	
	By \cref{Lem-Gammast1}, $y$ admits unique $x\in\p\Om,z\in\p\Om_t$. Note that $\underline{xy},\underline{yy'}$ and the interior angle between them, say $\alpha$, form an Hinge in $\mfS^{n+1}$. Now we consider a hinge in the Euclidean space, with the same lengths ${\rm dist}_g(x,y), {\rm dist}_g(y,y')$ and the interior angle $\alpha$. Note that a Euclidean hinge indeed induces a triangle, and we denote by $c$ the length of the other segment of this triangle, by the cosine theorem, we have
	\begin{align}\label{eq-hinge1}
	    c^2
	    =&{\rm dist}_g(x,y)^2+{\rm dist}_g(y,y')^2-2{\rm dist}_g(x,y){\rm dist}_g(y,y')\cos\alpha\notag\\
	    =&s^2+{\rm dist}_g(y,y')^2-2s\cdot{\rm dist}_g(y,y')\cos\alpha.
	\end{align}
	Using the hinge version of Toponogov's comparison theorem, see for example \cite[Theorem 12.2.2]{Petersen16}, we find
	\begin{align}\label{eq-hinge2}
	    {\rm dist}_g^2(x,y')^2\leq c^2.
	\end{align}
	By virtue of \eqref{eq-N-xy}, we can compute the interior angle $\alpha$ of $\underline{xy}$ and $\underline{yy'}$ at $y$, which is given by
	\begin{align*}
	    \cos\alpha=-N(y)\cdot\left(-\tilde N(y)\right)=N(y)\cdot\left(-\frac{1}{\sin\left({\rm dist}_g(y,y')\right)}y'+\cot\left({\rm dist}_g(y,y')\right)y\right),
	\end{align*}
	where $-\tilde N(y)$ denotes the initial velocity of the geodesic segment $\underline{yy'}$, which is a tangent vector at $y$. Combining this with \eqref{eq-hinge1} and \eqref{eq-hinge2}, we find 
	\begin{align*}
		{\rm dist}^2_g(x,y')
		\leq s^2+{\rm dist}_g^2(y,y')-2s\cdot{\rm dist}_g(y,y')N(y)\cdot\left(-\frac{1}{\sin\left({\rm dist}_g(y,y')\right)}y'+\cot\left({\rm dist}_g(y,y')\right)y\right),
	\end{align*}
	notice that ${\rm dist}_g(x,y')\geq s, N(y)\cdot y=0$, and hence we obtain
	\begin{align*}
		-2s\frac{{\rm dist}_g(y,y')}{\sin{\left({\rm dist}_g(y,y')\right)}}N(y)\cdot(y'-y)\leq{\rm dist}_g^2(y,y'),
	\end{align*}	
	since ${\rm dist}_g(y,y')<\pi$, we deduce that
	\begin{align}\label{N(y)(s)}
		N(y)\cdot(y'-y)\geq-\frac{1}{2s}\sin{\left({\rm dist}_g(y,y')\right)}{\rm dist}_g(y,y').
	\end{align}	
On the other hand, same computation holds for $y,y',z$ hold, notice that the interior angle in the geodesic triangle $y'yz$ at $y$ is given by $\cos\beta=N(y)\cdot(-\tilde N(y))$, thus we obtain
	\begin{align*}
		{\rm dist}_g^2(z,y')\leq(t-s)^2+{\rm dist}_g^2(y,y')
		+2(t-s)\cdot{\rm dist}_g(y,y')N(y)\cdot\left(-\frac{1}{\sin\left({\rm dist}_g(y,y')\right)}y'+\cot\left({\rm dist}_g(y,y')\right)y\right),
	\end{align*}
	notice that ${\rm dist}_g(y',z)\geq(t-s), {\rm dist}_g(y,y')<\pi$, we deduce
	\begin{align}\label{N(y)(t-s)}
		N(y)\cdot(y'-y)\leq\frac{1}{2(t-s)}\sin{\left({\rm dist}_g(y,y')\right)}{\rm dist}_g(y,y').
	\end{align}
	By \eqref{N(y)(s)} and \eqref{N(y)(t-s)} we obtain
	\begin{align}\label{Estimate-N(y)(y'-y)}
		\left\vert N(y)\cdot(y'-y)\right\vert\leq\max{\left\{\frac{1}{2s},\frac{1}{2(t-s)}\right\}}\sin{\left({\rm dist}_g(y,y')\right)}{\rm dist}_g(y,y').
	\end{align}
Since $u$ is differentiable along $\Gamma_s^t$, we see that $N$ is continuous on $\Gamma_s^t$.
Observe also
	\begin{align}\label{C^1Whitney}
		&\limsup_{\delta\ra0^+}\{\frac{\vert u(y')-u(y)-N(y)\cdot(y'-y)\vert}{\vert y'-y\vert}:0<\vert y'-y\vert\leq\delta,\quad y',y\in\Gamma_s^t\}\notag\\
		\leq&\limsup_{\delta\ra0^+}\left\{\frac{\max\{\frac{1}{2(t-s)},\frac{1}{2s}\}\sin({\rm dist}_g(y,y'))\cdot{\rm dist}_g(y,y')}{\vert y'-y\vert}:0<\vert y'-y\vert\leq\delta,y',y\in\Gamma_s^t\right\}
		=0,
	\end{align}
	where in the inequality we used the fact that $u(y')=u(y)=s$ and 
	\eqref{Estimate-N(y)(y'-y)}, in the equality we used the fact that as $\delta\ra0^+, {\rm dist}_g(y,y')\ra\vert y'-y\vert$ and also $\sin({\rm dist}_g(y,y'))\ra\vert y'-y\vert.$
	
	For $(u,N)\in C^0(\Gamma_s^t;\mfR\times\mfR^{n+2})$, since \eqref{C^1Whitney} holds, we may use the $C^1$-Whitney's extension theorem (see for example \cite[Section 15.2]{Mag12}) to see that there exists $\phi\in C^1(\mfR^{n+2})$ such that $(\phi,\nabla\phi)=(u,N)$ on $\Gamma_s^t$.
	
	For a fixed $y\in\Gamma_s^t$, we know that $N(y)\neq0$ by \cref{Lem-xyz-N}.
	Let $\{e_1,\ldots,e_{n+2}\}$ be the coordinate of $\mfR^{n+2}$, up to a rotation, we may assume that $y=(0,\ldots,0,1,0)=\nu_{\mfS^{n+1}}(y), N(y)=(0,\ldots,0,0,1)$. Since $\Gamma_s^t\subseteq \phi^{-1}(s)\cap\mfS^{n+1}$, we consider the following system 
	\begin{align*}
		\begin{cases}
			&f_1(x_1,\ldots,x_{n+2})=x_1^2+\ldots+x_{n+2}^2=1,\\
			&f_2(x_1,\ldots,x_{n+2})=\phi(y)=s.
		\end{cases}
	\end{align*}
	Notice that $ \nu_{\mfS^{n+1}}(y)=(0,\ldots,0,1,0),N(y)=(0,\ldots,0,0,1)$, and hence we have
	\begin{align*}
		&\p_{e_{n+1}}f_1(y)=1,\p_{e_{n+2}}f_1(y)=0,\\
		&\p_{e_{n+1}}f_2(y)=0,\p_{e_{n+2}}f_2(y)=1.
	\end{align*}
	Set $F:\mfR^n\times\mfR^2\ra\mfR^2$ by $F(x',x_{n+1},x_{n+2})=\left(f_1(x',x_{n+1},x_{n+2}),f_2(x',x_{n+1},x_{n+2})\right)$, then by the $C^1$-Implicit function theorem, there exists an open set $\mcU\subset\mfR^n$ and a $C^1$-map $\psi\in C^1(\mcU;\mfR^{2})$ such that $\Gamma_s^t\subset(x',\psi(x'))$ near $y$, i.e., $\Gamma_s^t$ lies in the $C^1$-image of $\Psi:\mcU\subset\mfR^n\ra\mfR^{n+2}$, given by $\Psi(x')=(x',\psi(x'))$. In particular, this shows the locally $\mcH^n$-rectifiability of $\Gamma_s^t$.
	\vspace{0.3cm}
	
	\noindent{\bf Step2. $C^{1,1}$-rectifiability of $\Gamma_s^+$.}
	
	{\color{black}Let $\mcC(N_1,N_2,\rho):=\left\{z+h_1N_1+h_2N_2:z\in {\rm span}\left\{N_1,N_2\right\}^\perp,\vert z\vert<\rho,\vert h_i\vert<\rho\right\}$ be the codimension-2 open cylinder at the origin with axis along $N_1,N_2\in T\mfS^{n+1}$, radius $\rho$ and height $2\rho$ in $\mfR^{n+2}$.
	By the fact that at any $y\in\Gamma_s^t$, $\left\{y\right\}=\p\mcB_{t-s}(z)\cap\p\mcB_s(x)$, $\nu_{\mfS^{n+1}}(y)=y$ and $\Gamma_s^t$ is locally $\mcH^n$-rectifiable, we have: $\Gamma_s^t$ admits an approximate tangent plane at $\mcH^n$-a.e. of its points and this plane is then exactly ${\rm span}\left\{N(y),\nu_{\mfS^{n+1}(y)}\right\}^\perp$, which is a $n$-dimensional affine plane in $\mfR^{n+2}$, i.e., 
	\begin{align*}
		T_y\Gamma_s^t={\rm span}\left\{N(y),y\right\}^\perp \quad\text{for }\mcH^n\text{-a.e. }y\in\Gamma_s^t. 
	\end{align*}
	By \cite[Theorem 10.2]{Mag12} and notice that for any fixed $\rho$, there exists $0<\rho_1<\rho_2$ such that $B_{\rho_1}\subset\mcC(N(y),y,\rho)\subset B_{\rho_2}$, we have
	\begin{align*}
		\lim_{\rho\ra0^+}\frac{\mcH^n\left(\Gamma_s^t\cap\left(y+\mcC\left(N\left(y\right),y,\rho\right)\right)\right)}{\om_n\rho^n}=1,\quad \text{for }\mcH^n\text{-a.e. }y\in\Gamma_s^t,
	\end{align*}
	here $\om_n$ denotes the volume of $n$-dimensional unit ball in $\mfR^{n+2}$.
	
	For a sequence $\{\rho_j\}_{j=1}^\infty$ such that $\rho_j\ra0$ as $j\ra\infty$, we set
	\begin{align*}
		f_j(y):=\frac{\mcH^n\left(\Gamma_s^t\cap\left(y+\mcC\left(N\left(y\right),y,\rho_j\right)\right)\right)}{\om_n\rho_j^n},
	\end{align*}
	then {\color{black}$f_j\ra1$} for $\mcH^n$-a.e. $y\in\Gamma_s^t$. By Egoroff's theorem and \cite[Lemma 1.1]{EG15}, there exists a compact set $\mcU_1\subset\Gamma_s^t$ such that $f_j\ra1$ uniformly on $\mcU_1$ with $\mcH^n(\Gamma_s^t\setminus \mcU_1)<\frac{1}{2}\mcH^n(\Gamma_{s}^t)$. For $\Gamma_s^t\setminus \mcU_1$, we may use Egoroff's theorem again to find a compact set $\mcU_2\subset\Gamma_s^t\setminus \mcU_1$ such that $f_j\ra1$ uniformly on $\mcU_2$ and $\mcH^n\left( \Gamma_s^t\setminus\left(\mcU_1\bigcup \mcU_2\right)\right)<\frac{1}{2^2}\mcH^n(\Gamma_s^t)$.
	By an inductive argument, we obtain a sequence of compact sets $\left\{\mcU_j\right\}_{j=1}^\infty$ such that $\mcH^n(\Gamma_s^t\setminus\bigcup_{j=1}^\infty\mcU_j)=0$ with $f_j\ra1$ uniformly on each $\mcU_j$, namely, 
	\begin{align}\label{mu_jstar}
		\mu_j^\ast(\rho):=\sup_{y\in U_j}\left\vert 1-\frac{\mcH^n\left(\Gamma_s^t\cap\left(y+\mcC\left(N\left(y\right),y,\rho\right)\right)\right)}{\om_n\rho^n}\right\vert\ra0 \quad\text{as }\rho\ra0^+.
	\end{align}
	This shows that $\Gamma_{s}^t$ can be covered by a countable union of compact sets, up to a $\mcH^n$-negligible set.
	
For every $j$ and for any $y\in \mcU_j\subset\Gamma_s^t$,
we know from the Implicit function theorem that $\mcU_j\subset\Gamma_s^t$ is contained in the graph of a $C^1$-map $\psi_j(\cdot)=(\psi_j^1(\cdot),\psi_j^2(\cdot)):\mfR^n\ra\mfR^2$ in a neighborhood of $y$.
Thanks to \eqref{mu_jstar} and \cref{Lem-appendix-2},
we may assume that, up to subdivision, rotation (so that $y=(0,\ldots,1,0)$ and $N(y)=(0,\ldots,0,1)$) of $\mcU_j$ and relabeling, there exists 
	\begin{align}\label{psi_j}
		\rho_j>0,\psi_j\in C^1({\rm span}\left\{N(y),y\right\}^\perp),
		\psi_j^1(0)=1,
		\psi_j^2(0)=0, \nabla_{z}\psi_j^i(0)=\vec{0},	
		\vert\nabla_{z}\psi_j^i\vert\leq 1
	\end{align}
	such that: let $V_j$ denote the projection of $\mcU_j$ on ${\rm span}\left\{N(y),y\right\}^\perp\cap\{\left\vert z\right\vert<\rho_j\}$, then
	\begin{align}\label{U_j-as-localgraph}
		\mcU_j\cap\left(y+\mcC\left(N\left(y\right),y,\rho_j \right) \right)=\Gamma_s^t\cap\left(y+\mcC\left(N\left(y\right),y,\rho_j \right) \right)=y+\left\{z+\psi_j^1(z)y+\psi_j^2(z)N(y): z\in V_j\right\},
	\end{align}
	here $\rho_j, \psi_j$ depend on the choice of $y\in \mcU_j$.
	Such $\mcU_j$ satisfies that: if we set 
	\begin{align}\label{eq-muj}
	    \mu_j(\rho)=\max\left\{\mu^\ast_j(\rho),\max_{\vert z\vert\leq\rho}\vert\nabla\psi_j^i(z)\vert\right\},\quad\rho\in(0,\rho_j],
	\end{align}
	then $\mu_j(\rho)\ra0$ as $\rho\ra0^+$ by virtue of \eqref{mu_jstar} and the continuity of $\nabla \psi_j^i(i=1,2)$. In the rest of the proof, we use $C_j$ to denote positive constants that depend only on $\mcU_j$.
	
We want to show that $N(y)$ is Lipschitz on each $\mcU_j$, namely,
for some constants $C_j$, it holds that
	\begin{align}\label{N-Lip-On-Uj}
		\left\vert N(y_1)-N(y_2)\right\vert\leq C_j\left\vert y_1-y_2\right\vert, \quad\text{for all }y_1, y_2\in \mcU_j.
	\end{align}
	
To have a chance to prove this, let us first point out that it suffice to consider the case when $y_1, y_2$ are close enough. Precisely, for $r_j<\rho_j/3$, we may assume that 
\begin{align}\label{eq-rectifiability-y1y2}
    y_1\in y_2+\mcC(N(y_2),y_2,\rho_j)
\end{align}
or otherwise, $\vert y_1-y_2\vert\geq c(n)r_j$ and it is trivial to see that $\vert N(y_1)-N(y_2)\vert\leq 2\leq C_j\vert y_1-y_2\vert$.

Next, with \eqref{eq-rectifiability-y1y2}, we may further assume, up to a rigid motion as before, that
\begin{align}\label{eq-y2Ny2}
    y_2=(\vec{0},1,0)\in\mfR^n\times\mfR^2,\quad N(y_2)=(\vec{0},0,1)\in\mfR^n\times\mfR^2.
\end{align}
In this way, \eqref{U_j-as-localgraph} reads as
\begin{align}\label{eq-graphy2}
    \left\{(z,h_1,h_2)\in\Gamma_s^t:\vert z\vert<\rho_j,\vert h_i\vert<\rho_j\right\}=\left\{(z,\psi_j^1(z),\psi_j^2(z)):z\in V_j\right\},
\end{align}
with $\psi^i_j\in C^1(V_j)$ satisfying \eqref{psi_j}.

By \eqref{eq-rectifiability-y1y2} again, $y_1=(z_1,\psi_j^1(z_1),\psi_j^2(z_1))$ for some $z_1\in V_j$ with $\vert z_1\vert<r_j$. Since $\psi_j^i$ is $C^1$ on $V_j$, the Taylor theorem gives
\begin{align}
    \psi_j^i(z_1)
    ={\color{black}\psi_j^i(\vec0)+o(\vert z_1\vert)}.
\end{align}
In view of this and invoking \eqref{psi_j}, \eqref{appendix-eq-1} thus gives a normal vector field in the form
\begin{align}\label{defn-tildeN}
    \tilde N(y_1)=&\left(\left(-1+o(\vert z_1\vert)+\left<z_1,\nabla_z\psi_j^2\right>\right)\nabla_z\psi_j^2+o(\vert z_1\vert)\nabla_z\psi_j^1,o(\vert z_1\vert),1-o(\vert z_1\vert)\right)\notag\\
    =&\left(\left(-1+o(\vert z_1\vert)\right)\nabla_z\psi_j^2+o(\vert z_1\vert)\nabla_z\psi_j^1,o(\vert z_1\vert),1-o(\vert z_1\vert)\right),
\end{align}
where $\vert \nabla_z\psi_j^1\vert\leq 1$ on each $V_j$ due to \eqref{psi_j}. In particular, set $N(y)=\tilde N(y)/\vert\tilde N(y)\vert$ and we readily see that
\begin{align}\label{ineq-N(y1)-z1}
    \vert N(y_1)-(\vec0,0,1)\vert^2\leq C_j\vert z_1\vert^2,
\end{align}
once provided
\begin{align}\label{eq-psi2}
    \vert\nabla_z\psi_j^2\vert\leq C_j\vert z\vert\quad\text{on }V_j.
\end{align}
Let us verify the validity of \eqref{eq-psi2} by Delgadino-Maggi's approach (see \cite[proof of (3-25)]{DM19} for a detailed codimension-1 argument).
{\color{black}First, for any $y,y_0\in y_2+\mcC(N(y_2),y_2,\rho_j)$ as in \eqref{eq-graphy2},
we set
\begin{align}
    N_1(y)=\frac{(-\nabla_z\psi_j^1,1)}{\sqrt{1+\vert\nabla_z\psi_j^1\vert^2}}, y'=(z,\psi_j^1(z)),y_0'=(z_0,\psi_j^2(z_0)).\\
    N_2(y)=\frac{(-\nabla_z\psi_j^2,1)}{\sqrt{1+\vert\nabla_z\psi_j^2\vert^2}}, y''=(z,\psi_j^2(z)),y_0''=(z_0,\psi_j^2(z_0)).\label{eq-N2y''y0''}
\end{align}
Note that similar with \eqref{eq-N(y)-linearCombination}, we may write a normal vector field as
\begin{align}\label{eq-Na1a2}
    \tilde N(y)=a_1(y)\frac{(-\nabla_z\psi_j^1,1,0)}{\sqrt{1+\vert\nabla_z\psi_j^1\vert^2}}+a_2(y)\frac{(-\nabla_z\psi_j^2,0,1)}{\sqrt{1+\vert\nabla_z\psi_j^2\vert^2}},
\end{align}
where $a_1(y)=\left(\left<z,\nabla_z\psi^2\right>-\psi^2(z)\right)\sqrt{1+\vert\nabla_z\psi_j^1\vert^2}, a_2(y)=\left(-\left<z,\nabla_z\psi^1\right>+\psi^1(z)\right)\sqrt{1+\vert\nabla_z\psi_j^2\vert^2}$.

Observe that $a_1(y)=-\left<N_2(y),y''\right>\sqrt{1+\vert\nabla_z\psi_j^1\vert^2}\sqrt{1+\vert\nabla_z\psi_j^2\vert^2}$, and from \eqref{psi_j} we know that $$a_2((\vec 0,1,0))=1,
\frac{(-\nabla_z\psi_j^1,1,0)}{\sqrt{1+\vert\nabla_z\psi_j^1\vert^2}}\mid_{z=\vec0}=(\vec0,1,0)=y\mid_{z=\vec0},$$
by continuity, we may assume that $a_2(y),\left<(-\nabla_z\psi_j^1,1,0),y_0\right>$ are closed to $1$, and we have
\begin{align}\label{eq-tildeN-y0}
    \left<\tilde N(y),y_0\right>
    =&-\left<N_2(y),y''\right>\left<(-\nabla_z\psi_j^1,1,0),y_0\right>+\left<N_2(y),y_0''\right>\left(-\left<z,\nabla_z\psi^1\right>+\psi^1(z)\right)\notag\\
    =&\left<N_2(y),y_0''-y''\right>\left(-\left<z_0,\nabla_z\psi_j^1\right>+\psi^1_j(z_0)\right)\notag\\
    &+\left<N_2(y),y_0''\right>\left(\psi_j^1(z)-\psi_j^1(z_0)+\left<z_0-z,\nabla_z\psi_j^1(z)\right>\right),
\end{align}
for the second term, since $\Gamma_s^t$ is trapped between two mutually tangent geodesic balls at $y$,
we may project $\Gamma_s^t$ and the two geodesic balls over the plane $T_{y_2}\mfS^{n+1}$ to find that the projected graph is punctually second order differentiable at $z$ and
\begin{align*}
    \psi_j^1(z)-\psi_j^1(z_0)+\left<z_0-z,\nabla_z\psi_j^1(z)\right>
    =o(\vert z-z_0\vert^2),
\end{align*}
which is controlled by $o(\vert y-y_0\vert^2)$.
On the other hand, by continuity we may assume that $\vert\tilde N(y)\vert\leq 2$ and hence \eqref{eq-tildeN-y0} together with
\eqref{Estimate-N(y)(y'-y)} yields that
\begin{align}\label{estimate-N2}
    C_j\vert y-y_0\vert^2\geq
    \left\vert \left<N_2(y),y''-y_0''\right>\right\vert.
\end{align}
}
We exploit \eqref{estimate-N2}
in the manner of Delgadino-Maggi, with $y=y_1$ and $y_0=(z_0,h_1^0,h_2^0)$, defined by
\begin{align}
    z_0=z_1-\vert z_1\vert e_0,\quad h_1^0=\psi_j^1(z_0),\quad h_2^0=\psi_j^2(z_0),
\end{align}
where $e_0=-\frac{\nabla_z\psi_j^2(z_1)}{\vert\nabla_z\psi_j^2(z_1)\vert}$ is a unit vector, determined as in \cite[(3-30)]{DM19}. \eqref{estimate-N2} then gives
\begin{align}\label{y1-y0}
    C_j\vert y_1-y_0\vert^2\geq\vert a_2(y_1)\vert\left<N_2(y_1),y_1''-y_0''\right>
    =\vert z_1\vert\frac{\left<\nabla_z\psi_j^2(z_1),-e_0\right>}{\sqrt{1+\vert\nabla_z\psi_j^2(z_1)\vert^2}}+\frac{\psi_j^2(z_1)-\psi_j^2(z_0)}{\sqrt{1+\vert\nabla_z\psi_j^2(z_1)\vert^2}}.
\end{align}
To proceed, let us note that for all $\vert z\vert<\rho_j$ such that $(z,\psi_j^1(z),\psi_j^2(z))\in\Gamma_s^t$, it holds that
\begin{align}\label{eq-estimate-psi}
    \vert\psi_j^2(z)\vert\leq C_j\vert z\vert^2,
\end{align}
this is a direct consequence of the following fact: near $y_2=(\vec0,1,0)$, at every $y\in\Gamma_s^t$, $\Gamma_s^t$ is trapped between two mutually tangent geodesic balls.
Due to this, we note that we can only use the estimate \eqref{eq-estimate-psi} for those points that lies in $\Gamma_s^t$. 

By definition of $z_0$, we have $\vert z_0\vert\leq2\vert z_1\vert<2r_j<\rho_j$, if $y_0$ lies exactly in $\Gamma_s^t\in \mcU_j$,
by \eqref{eq-estimate-psi} and the definition of $y_0$, we find
\begin{align}
    &\vert y_1-y_0\vert^2=\vert z_1\vert^2+\sum_{i=1}^2(\psi_j^i(z_1)-\psi_j^i(z_0))^2\leq C_j\vert z_1\vert^2,\\
    &\left\vert\frac{\psi_j^2(z_1)-\psi_j^2(z_0)}{\sqrt{1+\vert\nabla_z\psi_j^2(z_1)\vert^2}}\right\vert\leq\vert\psi_j^2(z_1)\vert+\vert\psi_j^2(z_0)\vert\leq C_j\vert z_1\vert^2,
\end{align}
and hence, recalling the definition of $e_0$, \eqref{y1-y0} gives
\begin{align}
    C_j\vert z_1\vert^2\geq\vert z_1\vert \vert \nabla_z\psi_j^2(z_1)\vert,
\end{align}
this shows \eqref{eq-psi2} when $y_0\in\Gamma_s^t$. On the other hand, if $y_0\notin\Gamma_s^t$, we let $\ep_0$ be the largest $\ep>0$ such that 
\begin{align}
    \left\{\vert z-z_0\vert<\ep\right\}\cap V_j=\varnothing.
\end{align}
Since $z_1\in V_j$ and $\vert z_0-z_1\vert=\vert z_1\vert$ by definition, we know that $\ep_0\leq\vert z_1\vert$. Moreover, since $\vert z_0\vert\leq2\vert z_1\vert$ by definition of $z_0$, it follows that the $n$-dimensional ball $\left\{\vert z-z_0\vert<\ep_0\right\}$ is contained in $\left\{\vert z\vert<3\vert z_1\vert\right\}\subset\left\{\vert z\vert<\rho_j\right\}$ thanks to $3r_j<\rho_j$. Our definition of $\ep_0$ then assures the existence of $z_\ast\in V_j$ with $\vert z_\ast-z_0\vert=\ep_0$ so that
\begin{align}\label{eq-z0-zstar}
    \om_n\vert z_0-z_\ast\vert^n
    =&\mcH^n(\left\{\vert z-z_0\vert<\ep_0\right\})\leq\mcH^n(\left\{\vert z\vert<3\vert z_1\vert\right\}\setminus V_j)\notag\\
    =&\om_n(3\vert z_1\vert)^n-\mcH^n(V_j\cap\left\{\vert z\vert<3\vert z_1\vert\right\}).
\end{align}
On the other hand,
the definition of $\mu_j$ in \eqref{eq-muj} shows that
\begin{align}\label{eq-muj3z1}
    \om_n(3\vert z_1\vert)^n(1-\mu_j(3\vert z_1\vert))\leq\mcH^n(\Gamma_s^t\cap \mcC(N(y_2),y_2,3\vert z_1\vert))\leq\om_n(3\vert z_1\vert)^n(1+\mu_j(3\vert z_1\vert)).
\end{align}
Moreover, recall that $\mcU_j$ is the graph of $\psi_j=(\psi_j^1,\psi_j^2)$ over $V_j$, and the Jacobian of the $C^1$-map $\Psi_j:z\mapsto(z,\psi_j^1(z),\psi_j^2(z))$ is
\begin{align}
    J_{\Psi_j}(z)=\sqrt{{\rm det}\left(\left<D_i\Psi_j,D_k\Psi_j\right>\right)_{1\leq i,k\leq n}},
\end{align}
where $D_i\Psi_j(z)=(0,\ldots,1,0\ldots,0,\p_i\psi_j^1(z),\p_i\psi_j^2(z))$ is the directional derivative of $\Psi_j$ along $e_i$, with $\left\{e_1,\ldots,e_n\right\}$ is the standard Euclidean coordinate of $\mfR^n$. We can use the Laplace expansion for the $n\times n$ matrix $(<D_i\Psi_j,D_k\Psi_j>)_{1\leq i,k\leq n}$ to see that 
\begin{align}
    J_{\Psi_j}(z)=\sqrt{1+\text{terms involving }(\p_i\psi_j^1(z),\p_k\psi_j^2(z))}.
\end{align}
In particular, by virtue of \eqref{psi_j},\eqref{eq-muj} and the continuity of $\nabla_z\psi_j^1, \nabla_z\psi_j^2$, we find: $J_{\Psi_j}(0)=1$, $J_{\Psi_j}(z)$ is close to $1$ near $z=0$, and hence non-vanishing on $V_j$.

Again, by definition of $\mu_j$ in \eqref{eq-muj}, we find
\begin{align}
    \mcH^n(V_j\cap\left\{\vert z\vert<3\vert z_1\vert\right\})
    =&\int_{V_j\cap\left\{\vert z\vert<3\vert z_1\vert\right\}}\frac{J_{\Psi_j}(z)}{J_{\Psi_j(z)}}
    \geq\frac{\int_{V_j\cap\left\{\vert z\vert<3\vert z_1\vert\right\}}J_{\Psi_j}(z)}{\sqrt{1+C(n)\mu_j(3\vert z_1\vert)^{2n}}}\notag\\
    =&\frac{\mcH^n(\Gamma_s^t\cap\mcC(N(y_2),y_2,3\vert z_1\vert))}{\sqrt{1+C(n)\mu_j(3\vert z_1\vert)^{2n}}}\geq\frac{1-\mu_j(3\vert z_1\vert)}{\sqrt{1+C(n)\mu_j(3\vert z_1\vert)^{2n}}}\om_n(3\vert z_1\vert)^n,
\end{align}
where we have used \eqref{eq-muj} in the Laplace expansion of $J_{\Psi_j}(z)$ for the first inequality, the area formula for $\Psi_j$ in the second equality, and \eqref{eq-muj3z1} for the last inequality.

Putting this estimate back into \eqref{eq-z0-zstar}, we find
\begin{align}
    \om_n\vert z_0-z_\ast\vert^n\leq C\mu_j(3\vert z_1\vert)\om_n(3\vert z_1\vert)^n;
\end{align}
i.e.,
\begin{align}\label{ineq-z0-zast}
    \vert z_0-z_\ast\vert\leq C\mu_j(3\vert z_1\vert)^{1/n}\vert z_1\vert.
\end{align}
It follows that $\vert z_\ast\vert\leq\vert z_0\vert+\vert z_0-z_\ast\vert\leq C\vert z_1\vert$, since we know from the definition of $z_0$ that $\vert z_0\vert\leq 2\vert z_1\vert$.

The definition of $z_\ast$ implies that $y_\ast=(z_\ast,\psi_j^1(z_\ast),\psi_j^2(z_\ast))\in \Gamma_s^t$, and hence we may apply \eqref{estimate-N2} with $y_1,y_\ast$ to obtain
\begin{align}\label{ineq-y1-yast-1}
    C_j\vert y_1-y_\ast\vert^2
    \geq&\left<N_2(y_1),y_1''-y_\ast''\right>\notag\\
    \geq&\frac{\left<-\nabla_z\psi_j^2(z_1),z_1-z_\ast\right>}{\sqrt{1+\vert\nabla_z\psi_j^2(z_1)\vert^2}}+\frac{\psi_j^2(z_1)-\psi_j^2(z_\ast)}{\sqrt{1+\vert\nabla_z\psi_j^2(z_1)\vert^2}}\notag\\
    \geq&\frac{\left<-\nabla_z\psi_j^2(z_1),z_1-z_\ast\right>}{\sqrt{1+\vert\nabla_z\psi_j^2(z_1)\vert^2}}-C(\vert z_1\vert^2+\vert z_\ast\vert^2)\notag\\
    \geq&\vert \nabla_z\psi_j^2(z_1)\vert(1-C\mu_j(3\vert z_1\vert)^{1/n})\frac{\vert z_1\vert}{C}-C(\vert z_1\vert^2+\vert z_\ast\vert^2),
\end{align}
where we have used \eqref{eq-estimate-psi} in the third inequality thanks to the fact that $y_1,y_\ast\in\Gamma_s^t$; for the last inequality, we first decomposed $z_1-z_\ast$ into the sum of $z_1-z_0=\vert z_1\vert e_0=-\vert z_1\vert\frac{\nabla _z\psi_j^2(z_1)}{\vert\nabla_z\psi_j^2(z_1)\vert}$ and $z_0-z_\ast$, then we used \eqref{ineq-z0-zast}.
Notice also
\begin{align}\label{ineq-y1-yast-2}
    \vert y_1-y_\ast\vert
    \leq&\vert z_1-z_\ast\vert+\sum_{i=1}^2\vert\psi_j^i(z_1)-\psi_j^i(z_\ast)\vert\notag\\
    \leq&\vert z_1-z_0\vert+\vert z_0-z_\ast\vert+C(\vert z_1\vert^2+\vert z_\ast\vert^2)\leq C\vert z_1\vert.
\end{align}
Thus,  by combining \eqref{ineq-y1-yast-1} with \eqref{ineq-y1-yast-2},
we have proved \eqref{eq-psi2} and hence \eqref{ineq-N(y1)-z1}. In particular, \eqref{ineq-N(y1)-z1} implies \eqref{N-Lip-On-Uj} immediately, since we have the trivial observation
\begin{align*}
    \vert y_1-y_2\vert^2=\vert z_1\vert^2+\vert \psi_j^1(z_1)\vert^2+\vert \psi_j^2(z_1)-1\vert^2\geq\vert z_1\vert^2.
\end{align*}
Therefore we have showed that $N$ is Lipschitz on each $\mcU_j$.}
By \eqref{Estimate-N(y)(y'-y)}
and \eqref{N-Lip-On-Uj}, on each $\mcU_j$, we can use the Whitney-Glaser extension theorem (see for example \cite{LGr09}) to find that there exists $\phi\in C^{1,1}(\mfR^{n+2})$ such that $(u,N)=(\phi,\nabla\phi)$ on $\mcU_j$.
	Then, by the  ${C^{1,1}}$-Implicit function theorem, for each $y\in \mcU_j$,
	there exists $\psi_j=(\psi_j^1,\psi_j^2)\in C^{1,1}({\rm span}\left\{N(y),y\right\}^\perp)$,
	which completes the proof.
\end{proof}

\section{The Heintze-Karcher inequality on sphere}\label{Sec-5}
With the $C^{1,1}$-rectifiability in force, we are going to derive the definitions of the principal curvature, viscosity mean curvature and boundary in the spirit of Delgadino-Maggi, thus extends \cite[Lemma 7 ]{DM19} from Euclidean space to $\mfS^{n+1}$. In this section, we continue to use the notations in \cref{defn-Gamma_s^t}, \eqref{defn-Omegas} and in the proof of \cref{Thm-C11Rec}.
\begin{proposition}\label{Prop3-summary}
	For any nonempty open set $\Om\subset\mfS^{n+1}$, there holds 
	\begin{enumerate}
		\item $N$ is tangentially differentiable along $\Gamma_s^t$ at $\mcH^n$-a.e. $y\in\Gamma_s^t$, with
		\begin{align}\label{eq-range-kappa}
		   \begin{cases}
		       \nabla^{\Gamma_s^t}N(y)=-\sum_{i=1}^n(\kappa_s^t)_i(y)\tau_i(y)\otimes\tau_i(y),\\
		   -\cot s\leq(\kappa_s^t)_i(y)\leq\cot(t-s),
		   \end{cases}
		\end{align}
			where $\left\{\left( \kappa_s^t\right)_i(y)\right\}_{i=1}^n$ denote the principal curvatures of $N$ along $\Gamma_s^t$ at $y$ when they exist, which are indexed in increasing order.
		\item For a.e. $0<s<\pi$, $\mcH^n(\Gamma_s^+)=\mcH^n(\p\Omega_s)$.
	\item \label{prop4.1-item3-new}For every $r<s<t$, the map $g_r:\Gamma_s^t\ra\Gamma_{s-r}^t$, given by $g_r(y)=\cos{r}y-\sin{r}N(y)$ for $y\in\Gamma_s^t$, is a bijection from $\Gamma_s^t$ to $\Gamma_{s-r}^t$ and is Lipschitz when restricted to $\mcU_j$, with
	\begin{align}
		J^{\Gamma_s^t}g_r(y)=\prod_{i=1}^n\left[\cos{r}+\sin{r}(\kappa_s^t)_i \right],\quad (\kappa_{s-r})_i(g_r(y))=\frac{-\sin{r}+\cos{r}(\kappa_s)_i(y)}{\cos{r}+\sin{r}(\kappa_s)_i(y)},
	\end{align}
for $\mcH^n$-a.e. $y\in\Gamma_s^t$.
	\end{enumerate}
\end{proposition}
\begin{proof}
We begin by recalling that in \cref{Thm-C11Rec}, we have constructed a sequence of compact sets $\mcU_j$, on which $N$ is Lipschitz (see \eqref{N-Lip-On-Uj}) such that $\mcH^n(\Gamma_s^t\setminus\bigcup_{j=1}^\infty \mcU_j)=0$.
\vspace{0.3cm}

{\color{black}\noindent{\bf (1)} By virtue of \cref{DM-2.1(iv)}, to study the tangential gradient of $N$ on $\Gamma_s^t$, it suffice to work on each $\mcU_j$, see \eqref{psi_j} and \eqref{U_j-as-localgraph} for the construction of $\mcU_j$, where $\mcU_j$ is contained in the graph of the $C^{1,1}$-map $(\psi_j^1(\cdot),\psi_j^2(\cdot))$.

Now, for a fixed $y\in\mcU_j$, we consider a natural Lipschitz extension of $N$,
from $\mcU_j\cap
\left(y+\mcC
\left(N\left(y\right),y,\rho_j \right) \right)$ to $y+\mcC(N(y),y,\rho_j)$, denoted by $N_\ast$ and defined as
	\begin{align}
	N_\ast(y+z+h_1y+h_2N(y))=N(y+z),\quad\forall z\in {\rm span}\left\{N(y),y\right\}^\perp,\vert z\vert,h_1,h_2
	<\rho_j,
	\end{align}
where $N(y+z)$ is just the normal of the graph $(x',\psi_j^1(x'),\psi_j^2(x'))$ at $y+z\in \mcU_j$.
Set $\Psi_j(z):=y+z+\psi_j^1(z)y+\psi^2_j(z)N(y)$ for $\vert z\vert<\rho_j$, {\color{black}by} \eqref{DM-2.5} we have: for $\mcH^n$-a.e. $y'\in \mcU_j$ and for any $\tau\in T_{y'}U_j$, 
	\begin{align*}
		\left(\nabla^{\mcU_j}N\right)_{y'}[\tau]=\nabla(N_\ast\circ \Psi_j)_{\Psi_j^{-1}(y')}[e],
	\end{align*}
	where $e=(\nabla \Psi_j)^{-1}_{\Psi^{-1}_j(y')}[\tau]\in\mfR^n$.
	
If $\psi_j\in C^2({\rm span}\left\{N(y),y\right\}^\perp)$, then for any $z\in {\rm span}\left\{N(y),y\right\}^\perp$, by definition,
	\begin{align*}
		\nabla(N_\ast\circ\Psi_j)_z[e]=\lim_{t\ra0^+}\frac{N_\ast(\Psi_j(z+te))-N_\ast(\Psi_j(z))}{t}=\lim_{t\ra0^+}\frac{N(y+te)-N(y)}{t},
	\end{align*}
this shows that
	\begin{align}\label{shapeoperator}
		\nabla(N_\ast\circ\Psi_j)_z[e]
		=-S_j(\Psi_j(z))[\tau],
	\end{align}
	where $S_j(\Psi_j)$ denotes the classical shape operator in differential geometry, with respect to the graph of $\psi_j=(\psi_j^1,\psi_j^2)$.}
	
Notice that $\Gamma_s^t$ is trapped between two mutually tangent geodesic balls on $\mfS^{n+1}$ with radius $s$ and $t-s$,
and hence the principal curvatures of the graph of $\psi_j$ is bounded from below by $-\cot{s}$ and from above by $\cot{(t-s)}$ , i.e.,
	\begin{align}\label{kappa_s^t}
		-\cot{s}\leq\left( \kappa_s^t\right)_i(y)\leq\left( \kappa_s^t\right)_{i+1}(y)\leq\cot{(t-s)}.
	\end{align} 

	By the chain rule for Lipschitz functions and the fact that $\nabla^2\psi_j$ is a $\mcH^n$-a.e. classical differential,
	the above argument holds for $\mcH^n$-a.e. $y\in \mcU_j$, {\color{black}which completes the proof of {\bf (1)}}.
	
	\vspace{0.3cm}
	\noindent{\bf (2)} 
	In view of \eqref{eq-N-yz}, given $r\in[-s,t-s]$, we consider the map $f_r:\Gamma_s^t\ra\p\Om_{s+r}$, defined by $f_r(y)=\cos{r}y+\sin{r}N(y)$. By definition of $\Gamma_s^t$, it is clear that $f_r$ is surjective, thus we have
	\begin{align*}
	\mcH^n(\p\Om_{s+r})=\mcH^n(f_r(\Gamma_s^t))\leq\int_{f_r(\Gamma_s^t)}\mcH^0(f^{-1}_r(z))\rd\mcH^n(z),
	\end{align*}
	using the area formula, we find
		\begin{align}\label{ineq-area1}
		\mcH^n(\p\Om_{s+r})\leq\int_{f_r(\Gamma_s^t)}\mcH^0(f^{-1}_r(z))\rd\mcH^n(z)
		=\int_{\Gamma_s^t}J^{\Gamma_s^t}f_r(y)\rd\mcH^n(y),
	\end{align}
	a direct computation then gives that, for $\mcH^n$-a.e. $y\in\Gamma_s^t$,
		\begin{align*}
		J^{\Gamma_s^t}f_r(y)=\prod_{i=1}^n\left[\cos{r}-\sin{r}(\kappa_s^t)_i \right].
	\end{align*}
	For $0<r<t-s$, we have, $\cot r>\cot(t-s)\geq(\kappa_s^t)_i(y)$ by virtue of \eqref{kappa_s^t}, it follows that $\cos r-\sin r(\kappa_s^t)_i>0$ for each $i$, and hence we can use the Cauchy-Schwarz inequality to find
		\begin{align*}
		J^{\Gamma_s^t}f_r(y)=\prod_{i=1}^n\left[\cos{r}-\sin{r}(\kappa_s^t)_i \right]\leq\left\{\left[ \cot{r}+\cot{s}\right]\sin{r}\right\}^n.
	\end{align*}
	It follows from \eqref{ineq-area1} that
		\begin{align}\label{ineq-area2}
	\mcH^n(\p\Om_{s+r})\leq\int_{\Gamma_s^t}\left\{\left[ \cot{r}+\cot{s}\right]\sin{r}\right\}^n\rd\mcH^n\leq\left\{\left[ \cot{r}+\cot{s}\right]\sin{r}\right\}^n\mcH^n(\Gamma_s^t).
	\end{align}
	On the other hand, by the Coarea formula, for a.e. $s>0$, we have
	\begin{align*}
	    \mcH^n(\p\Omega_s)=\lim_{\epsilon\ra0}\frac{1}{\epsilon}\int_0^\epsilon\mcH^n(\p\Omega_{s+r})\rd r,
	\end{align*}
		combining with \eqref{ineq-area1}, we obtain
	\begin{align*}
		\frac{1}{\ep}\int_0^\ep\mcH^n(\p\Om_{s+r})\rd r
		&\leq\frac{1}{\ep}\int_0^\ep\left\{\left[ \cot{r}+\cot{s}\right]\sin{r}\right\}^n\mcH^n(\Gamma_s^t)\rd r\\
		&\leq\frac{\mcH^n(\Gamma_s^+)}{\ep}\int_0^\ep\left[1+\sin{\ep}\vert\cot{s}\vert\right]^n\rd r\\
		&=\left[1+\sin{\ep}\vert\cot{s}\vert\right]^n\mcH^n(\Gamma_s^+).
	\end{align*}
	Notice that $\Gamma_s^+\subset\p\Om_s$, thus we deduce
	\begin{align*}
		\mcH^n(\Gamma_s^+)
	\leq\mcH^n(\p\Om_s)\leq\lim_{\ep\ra0}\left[1+\sin{\ep}\vert\cot{s}\vert\right]^n\mcH^n(\Gamma_s^+)=\mcH^n(\Gamma_s^+),
	\end{align*}
	which proves ${\bf(2)}$.
	\vspace{0.3cm}
	
	\noindent\textbf{(3)} For $r\in(0,s)$, consider the bijection $g_r:\Gamma_s^t\ra\Gamma_{s-r}^t$, defined by $g_r(y)=\cos{r}y-\sin{r}N(y)$, which is Lipschitz on each $\mcU_j$.
	We claim that, if $N$ is tangential differentiable at $y$ along $\Gamma_s^t$,
	then $N$ is also tangential differentiable at $g_r(y)$ along $\Gamma_{s-r}^t$.
	
	Indeed, by a simple geometric relation on sphere ({\color{black}as illustrated in \cref{Fig-2}}), we have
	\begin{align*}
		g_r(y)+\tan{\frac{r}{2}}N(g_r(y))=y-\tan{\frac{r}{2}}N(y),
	\end{align*}
which implies
	\begin{align*}
		N(g_r(y))
		=-\left[\left(\cos{r}-1\right)y-\left(\sin{r}-\tan{\frac{r}{2}}\right)N(y) \right]\cdot\frac{1}{\tan{\frac{r}{2}}}
		=\sin{r}y+\cos{r}N(y).
	\end{align*}
	Thus 
	\begin{align*}
		\cos{r}N(y)=-\sin{r}y+N(g_r(y))=-\sin ry+N\left(\cos ry-\sin rN(y)\right).
	\end{align*}
	\begin{figure}[h]
	\centering
	\includegraphics[height=7cm,width=11cm]{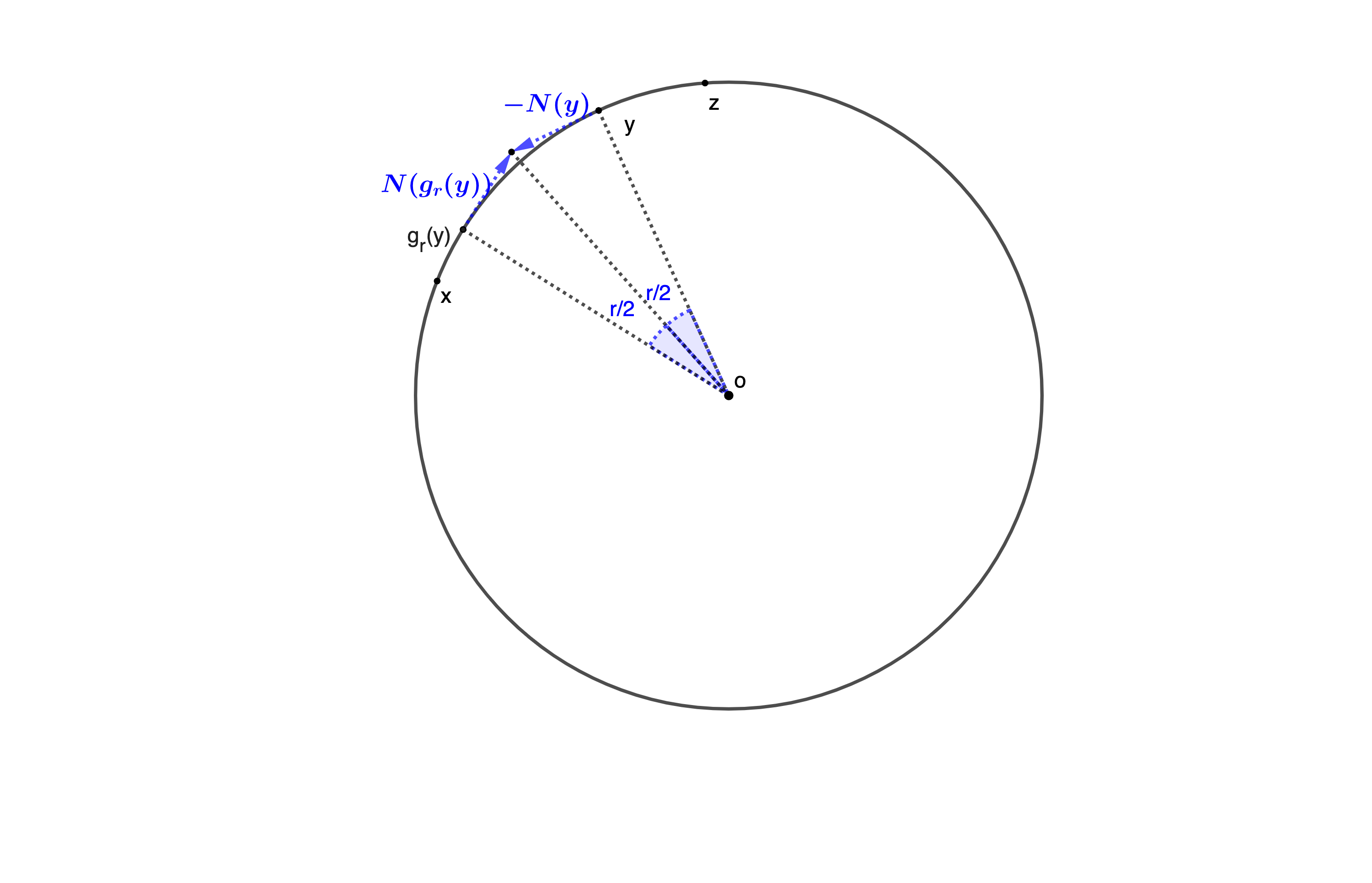}
	\caption{Relation of $y$, $g_r(y)$ and the corresponding normals.}
	\label{Fig-2}
\end{figure}
That is, if $y$ is a point of tangential differentiability of $N$ along $\Gamma_s^t$ and $\tau\in T_y\Gamma_s^t$, then $\tau\in T_{g_r(y)}\Gamma_s^t$ and
	\begin{align*}
		\cos{r}\left(\nabla^{\Gamma_s^t}N\right)_y\left[\tau\right]=-\sin{r}\tau+\left(\nabla^{\Gamma_{s-r}^t}N\right)_{g_r(y)}\left[\cos{r}\tau-\sin{r}\left(\nabla^{\Gamma_s^t}N\right)_y\left[\tau\right]\right],
	\end{align*}
	take $\tau=\tau_i(y)$ to be the eigenvectors of the shape operators $S_j$ in \eqref{shapeoperator}, we obtain
	\begin{align*}
		-\cos{r}(\kappa_s^t)_i(y)\tau_i(y)
		&=-\sin{r}\tau_i(y)+\left(\nabla^{\Gamma_{s-r}^t}N\right)_{g_r(y)}\left[\cos{r}\tau_i(y)+\sin{r}(\kappa_s^t)_i(y)\tau_i(y) \right]\\
		&=-\sin{r}\tau_i(y)+\left(\cos{r}+\sin{r}(\kappa_s^t)_i(y)\right)\left(\nabla^{\Gamma_{s-r}^t}N\right)_{g_r(y)}[\tau_i(y)],
	\end{align*}
	from this we have
	\begin{align*}
		-\tau_i(y)\cdot\left(\nabla^{\Gamma_{s-r}^t}N\right)_{g_r(y)}[\tau_i(y)]=\frac{-\sin{r}+\cos{r}(\kappa_s^t)_i(y)}{\cos{r}+\sin{r}(\kappa_s^t)_i(y)}.
	\end{align*}
	Hence $\{\tau_i(y)\}_{i=1}^n$ is an orthonormal basis for $T_{g_r(y)}\Gamma_{s-r}^t$, and the eigenvalues of $\nabla^{\Gamma_{s-r}^t}N(g_r(y))$ are given by:
	\begin{align}\label{kappa(s-r)}
		(\kappa_{s-r}^t)_i(g_r(y))=\frac{-\sin{r}+\cos{r}(\kappa_s^t)_i(y)}{\cos{r}+\sin{r}(\kappa_s^t)_i(y)},
	\end{align}
	which completes the proof of (3).
\end{proof}
\begin{remark}
    \normalfont
    In view of the rectifiability theorem \cref{Thm-C11Rec} and \cref{Prop3-summary}(2), we have proved the following: for any nonempty open set $\Omega\subset\mfS^{n+1}$, the level-set $\p\Omega_s$ of the distance function $u$ is $C^{1,1}$-rectifiable for a.e. $0<s<\pi$.
\end{remark}	
	Now we are in the position to generalize the viscosity mean curvature defined in \cite{DM19} from Euclidean space to $\mfS^{n+1}$.
	\begin{definition}[Principal curvature and second fundamental form on $\Gamma_s^+$]\label{Defn-2ndfundamentalform}
	    \normalfont For every $0<s<\pi$, the \textit{principal curvatures $(\kappa_s)_i$ of $\Gamma_s^+$} are defined $\mcH^n$-a.e. on $\Gamma_s^+$ by setting
		\begin{align*}
			(\kappa_s)_i=(\kappa_s^t)_i \quad\text{on }\Gamma_s^t\text{ for each }t>s.
		\end{align*}
		where $(\kappa_s^t)_i$ is well-defined on $\mcH^n$-a.e. $y\in\Gamma_s^t$ by virtue of \cref{Prop3-summary}{\bf (1)}.
		In particular, we can define the \textit{mean curvature and the length of the second fundamental form of $\p\Om_s$ with respect to $\nu_{\Om_s}$} at $\mcH^n$-a.e. points of $\Gamma_s^+$ as follows:
		\begin{align*}
			H_{\Om_s}=\sum_{i=1}^n(\kappa_s)_i,\quad\vert A_{\Om_s}\vert^2=\sum_{i=1}^n(\kappa_s)_i^2.
		\end{align*}
	\end{definition}
\begin{lemma}\label{Lem-PrincipalCurvature}
    For every $0<s\neq\pi/2<\pi$ and for $\mcH^n$-a.e. $y\in\Gamma_s^+$ where the principal curvatures are well-defined, let $x=g_s(y)$, then the limit
		\begin{align}\label{kappa_i Omega}
			\kappa_i(x)=\lim_{r\ra s^-}(\kappa_{s-r})_i(g_r(y))
		\end{align}  
		exists by monotonicity.
		\begin{proof}
		Let $y\in\Gamma_s^+$ be the point where the principal curvatures are well-defined, and let $x=g_s(y)$.
		For $0<r_1<r_2<s$, by \eqref{kappa(s-r)} we have
		\begin{align*}
			(\kappa_{s-r_1}^t)_i(g_{r_1}(y))-(\kappa_{s-r_2}^t)_i(g_{r_2}(y))
			&=\frac{-\tan{r_1}+(\kappa_s^t)_i(y)}{1+\tan{r_1}(\kappa_s^t)_i(y)}-\frac{-\tan{r_2}+(\kappa_s^t)_i(y)}{1+\tan{r_2}(\kappa_s^t)_i(y)}\\
			&=\frac{(\tan{r_2}-\tan{r_1})\cdot\left(1+(\kappa_s^t)_i^2(y)\right)}{\left(1+\tan{r_1}\left(\kappa_s^t\right)_i(y) \right)\cdot\left(1+\tan{r_2}\left(\kappa_s^t\right)_i(y) \right)}.
		\end{align*}
		 Notice also that $(\kappa_s^t)_i(y)$ is a bounded number as in \eqref{kappa_s^t}, and hence $(\kappa_{s-r}^t)_i(x)$ is monotone decreasing on $r$, it follows that \eqref{kappa_i Omega} exists.

		\end{proof}
\end{lemma}
	\begin{definition}\label{Defn-ViscosityMeanConvex}
		\normalfont
		For any nonempty open set $\Om$ in $\mfS^{n+1}$, the \textit{viscosity boundary} of $\Om$ is defined as
		\begin{align*}
			\p^v\Om=\bigcup_{s>0}g_s(\Gamma_s^+)
		\end{align*}
		and the corresponding \textit{viscosity mean curvature} of $\Om$ is defined at those $x$ where the limit \eqref{kappa_i Omega} exists for some $s>0$, defined by
		\begin{align*}
			H_{\Om}^v=\sum_{i=1}^n\kappa_i(x).
		\end{align*}
		\end{definition}
		We would like to mention that the viscosity principal curvatures
		are strictly related to the principal curvatures defined and investigated in \cite{Santilli20,HS22}.
		
		Finally, we can prove a Heintze-Karcher-type inequality in the spirit of Brendle's monotonicity approach \cite{Bre13}, see also \cite[Theorem 8]{DM19} for the Euclidean version.
\begin{proof}[Proof of \cref{Thm-HK}]
	We define for $0<s<\frac{\pi}{2}$,
	\begin{align*}
		Q(s)=\int_{\Gamma_s^+}\frac{\cos{s}}{H_{\Om_s}}\rd\mcH^n.
	\end{align*}
Notice that by the monotonicity (see \cref{Lem-PrincipalCurvature}), the viscosity mean convexity of $\Om$ implies $H_{\Om_s}>0$ on $\Gamma_s^+$ for every $s\in(0,\frac{\pi}{2})$.
With this observation, for every $s<t<\frac{\pi}{2}$, we define $Q^t:(0,t)\ra(0,\infty)$ by setting
\begin{align*}
	Q^t(s)=\int_{\Gamma_s^t}\frac{\cos{s}}{H_{\Om_s}}\rd\mcH^n.
\end{align*}
Observe that by definition,
\begin{align}\label{DM4-5}
	Q(s)\geq Q^t(s)\geq Q^{t+\ep}(s)\quad\text{for all }t>s,\ep>0,
\end{align}
notice also that $\mcH^n(\Gamma_s^t)$ converges monotonically to $\mcH^n(\Gamma_s^+)$ as $t\ra s^+$ by virtue of the inclusion \cref{Prop-Gammast}{\bf (1)}. This implies 
\begin{align}\label{DM4-6}
	Q(s)=\lim_{t\ra s^+}Q^t(s)=\sup_{t>s}Q^t(s)\quad\text{for all }0<s<\frac{\pi}{2}.
\end{align}
For $r\in(0,s)$, by \cref{Prop3-summary} \eqref{prop4.1-item3-new}, we have
\begin{align}\label{eqc1}
	Q^t(s-r)-Q^t(s)
	=&\int_{\Gamma_{s-r}^t}\frac{\cos{(s-r)}}{H_{\Om_{s-r}}}\rd\mcH^n-\int_{\Gamma_s^t}\frac{\cos{s}}{H_{\Om_s}}\rd\mcH^n\notag\\
	=&\int_{\Gamma_s^t}\left\{\frac{\cos(s-r)\prod_{i=1}^n\left[\cos{r}+\sin{r}(\kappa_s^t)_i\right]} {\sum_{i=1}^n(-\sin{r}+\cos{r}(\kappa_s^t)_i)/(\cos{r}+\sin{r}(\kappa_s^t)_i)}
	-\frac{\cos{s}}{H_{\Om_s}}\right\}\rd\mcH^n,
\end{align}
where 
\begin{align*}
	&\sum_{i=1}^n(-\sin{r}+\cos{r}(\kappa_s^t)_i)/(\cos{r}+\sin{r}(\kappa_s^t)_i)\notag\\
	=&\sum_{i=1}^n\frac{(-\sin{r}+\cos{r}(\kappa_s^t)_i)\prod_{j\neq i}(\cos{r}+\sin{r}(\kappa_s^t)_i)}
	{\prod_{i=1}^n\left[\cos{r}+\sin{r}(\kappa_s^t)_i\right]}\notag\\
	=&\frac{\left\{\cos^n{r}H_{\Om_s}+\cos^{n-1}{r}\sin{r}\left(H^2_{\Om_s}-\vert A_{\Om_s}\vert^2\right)-n\sin{r}\cos^{n-1}{r}+O(\sin^2{r})
		\right\}}{\prod_{i=1}^n\left[\cos{r}+\sin{r}(\kappa_s^t)_i\right]}.
\end{align*}
Thus \eqref{eqc1} reads
\begin{align*}
		&Q^t(s-r)-Q^t(s)\notag\\
		=&\int_{\Gamma_s^t}\left\{\frac{\cos(s-r)\left(\prod_{i=1}^n\left[\cos{r}+\sin{r}(\kappa_s^t)_i\right]\right)^2}{\cos^n{r}H_{\Om_s}+\cos^{n-1}{r}\sin{r}\left(H^2_{\Om_s}-\vert A_{\Om_s}\vert^2\right)-n\sin{r}\cos^{n-1}{r}+O(\sin^2{r})}
			-\frac{\cos{s}}{H_{\Om_s}}\right\}\rd\mcH^n\notag\\
		=&\int_{\Gamma_s^t}\left\{\frac{\cos(s-r)\left(\cos^{2n}{r}+2\sin{r}\cos^{2n-1}{r}H_{\Om_s}+O(\sin^2{r})\right)}{\cos^n{r}H_{\Om_s}+\cos^{n-1}{r}\sin{r}\left(H^2_{\Om_s}-\vert A_{\Om_s}\vert^2\right)-n\sin{r}\cos^{n-1}{r}+O(\sin^2{r})}
		-\frac{\cos{s}}{H_{\Om_s}}\right\}\rd\mcH^n.
\end{align*}
Notice that 
\begin{align*}
	\cos(s-r)\cos{r}=\cos{s}+\sin(s-r)\sin{r},
\end{align*}
and hence we have
\begin{align}
	&Q^t(s-r)-Q^t(s)\notag\\
	=&\int_{\Gamma_s^t}\left\{\frac{\cos^{2n-1}{r}\cos{s}+\cos^{2n-1}r\sin{(s-r)}\sin{r}+2\cos^{2n-2}r\cos{s}\sin{r}H_{\Om_s}+O(\sin^2{r})}{\cos^n{r}H_{\Om_s}+\cos^{n-1}{r}\sin{r}\left(H^2_{\Om_s}-\vert A_{\Om_s}\vert^2\right)-n\sin{r}\cos^{n-1}{r}+O(\sin^2{r})}
	-\frac{\cos{s}}{H_{\Om_s}}\right\}\rd\mcH^n\notag\notag,
\end{align}
where $O_t(\sin^2{r})/r\ra0$ uniformly on $\Gamma_s^t$ as $r\ra0$. We thus find $Q^t$ is differentiable on $(0,t)$ with
\begin{align*}
	(Q^t)'(s)=&\lim_{r\ra0}\frac{Q^t(s-r)-Q^t(s)}{-r}\notag\\
	=&-\int_{\Gamma_s^t}\cos{s}\left(1+\frac{\vert A_{\Om_s}\vert^2}{H^2_{\Om_s}}\right)\rd\mcH^n-\int_{\Gamma_s^t}\frac{n\cos{s}+H_{\Om_s}\sin{s}}{H_{\Om_s}^2}\rd\mcH^n.
\end{align*}
Notice that by \eqref{eq-range-kappa},
$(\kappa_s^t)_i\geq-\cot{s}$, which implies $H_{\Om_s}\sin{s}+n\cos{s}\geq0$ on $\Gamma_s^t$.
Also, by the Cauchy-Schwarz inequality we have $H^2_{\Om_s}\leq n\vert A_{\Om_s}\vert^2$, these facts imply
\begin{align}\label{DM4-7}
	(Q^t)'(s)\leq-\frac{n+1}{n}\int_{\Gamma_s^t}\cos{s}\rd\mcH^n.
\end{align}
For $0<s_1<s_2<\frac{\pi}{2}$, by \eqref{DM4-6}, \eqref{DM4-5} and \eqref{DM4-7} respectively, we find
\begin{align}\label{DM4-8}
	Q(s_1)-Q(s_2)
	=&\lim_{\ep\ra0^+}Q^{s_1+\ep}(s_1)-Q^{s_2+\ep}(s_2)\notag\\
	\geq&\lim_{\ep\ra0^+}Q^{s_2+\ep}(s_1)-Q^{s_2+\ep}(s_2)=Q^{s_2}(s_1)-Q^{s_2}(s_2)\notag\\
	\geq&\frac{n+1}{n}\int_{s_1}^{s_2}\left(\int_{\Gamma_s^{s_2}}\cos{s}\rd\mcH^n\right)\rd s=\frac{n+1}{n}\int_{s_1}^{s_2}\cos{s}\mcH^n(\Gamma_s^{s_2})\rd s,
\end{align}
in particular, $Q$ is decreasing on $(0,\frac{\pi}{2})$ and $Q'$ exists for a.e. $s$ by monotonicity. Using area formula, by virtue of \cref{Prop3-summary} \eqref{prop4.1-item3-new}, we have
\begin{align*}
	\mcH^n(\Gamma_{s-r}^t)=\int_{\Gamma_s^t}\prod_{i=1}^n\left[\cos{r}+\sin{r}(\kappa_s^t)_i \right]\rd\mcH^n,
\end{align*}
where $\left[\cos{r}+\sin{r}(\kappa_s^t)_i \right]\ra1$ uniformly on $\Gamma_s^t$ as $r\ra0$ by virtue of the fact that $-\cot{s}\leq(\kappa_s)_i\leq\cot{(t-s)}$, for each $i\in\{1,\ldots,n\}$. In particular, this shows $\mcH^n(\Gamma_s^t)$ is continuous on $s\in(0,t)$, and the mean value property yields
\begin{align*}
	\int_{s_1}^{s_2}\mcH^n(\Gamma_s^{s_2})\rd s
	=(s_2-s_1)\mcH^n(\Gamma_{s_0}^{s_2}),
\end{align*}
for some $s_0\in(s_1,s_2)$.

On the other hand, letting $r=t-s$ in \eqref{ineq-area2}, we find
\begin{align*}
		\mcH^n(\p\Om_{t})\leq\left\{\left[ \cot{(t-s)}+\cot{s}\right]\sin{(t-s)}\right\}^n\mcH^n(\Gamma_s^t),
\end{align*}
and it follows that
\begin{align*}
	\mcH^n(\p\Om_{s_2})\leq\liminf_{s\ra(s_2)^-}\left(\cos{(s_2-s)+\cot{s_2}\cdot\sin{(s_2-s)}}\right)^n\mcH^n(\Gamma_s^{s_2})\leq\liminf_{s\ra(s_2)^-}\mcH^n(\Gamma_s^{s_2}).
\end{align*}
From this observation, we deduce
\begin{align*}
	\liminf_{s_1\ra(s_2)^-}\frac{1}{s_2-s_1}\int_{s_1}^{s_2}\cos{s}\mcH^n(\Gamma_s^{s_2})\rd s
	\geq\cos{s}\mcH^n(\p\Om_{s_2})\quad\text{for all }0<s_2<\frac{\pi}{2}.
\end{align*}
Thus we conclude from \eqref{DM4-8} that 
\begin{align*}
	-Q'(s)\geq\frac{n+1}{n}\cos{s}\mcH^n(\p\Om_s)\quad\text{for a.e. }s>0.
\end{align*}
Finally, integrating this over $(s,\frac{\pi}{2})$, we obtain the Heintze-Karcher type inequality \eqref{HK-INEQ}. This completes the proof.
\end{proof}
\begin{remark}
    \normalfont
    We are informed by a referee that a general Heintze-Karcher inequality for closed subsets in a uniformly convex finite dimensional Banach space with a characterization of the equality case has been recently obtained in \cite{HS22}.
    
    It is natural to see if there is a characterization of the equality case in \eqref{HK-INEQ}, which could be useful for establishing an Alexandrov-type theorem on the standard sphere among sets of finite perimeter. 
\end{remark}



\appendix
{\color{black}
\section{Codimension-2 graphs}

The purpose of this appendix is to present some fundamental results for codimension-2 graphs restricted to the unit sphere in $\mfR^{n+2}$, that is convenient for this paper, since the computation is not easily found in other literatures.

Let $V\subset\mfR^n$ be an open set, given functions $\psi^1, \psi^2\in C^1(V)$. In all follows we use $z$ to denote points in $\mfR^n$, $\nabla_{z}$ denotes the gradient operator in $\mfR^n$. Consider the codimension-2 graph $G=\left\{\left(z,\psi^1(z),\psi^2(z)\right):z\in V\right\}$, we use $y\in\mfR^{n+2}$ to denote the points on this graph, i.e., for $y\in G$, there exists $z\in V$ such that $y=\left(z,\psi^1(z),\psi^2(z)\right)$.
\begin{lemma}\label{Lem-appendix-1}
    If the codimension-2, $C^1$-graph $G$ lies in $\mfS^{n+1}\subset\mfR^{n+2}$, then 
    \begin{align}\label{appendix-eq-1}
        \tilde N(y)=\left(
        \left<z,\nabla_{z}\psi^1\right>\nabla_{z}\psi^2-\left<z,\nabla_{z}\psi^2\right>\nabla_{z}\psi^1+\psi^2\nabla_{z}\psi^1-\psi^1\nabla_{z}\psi^2,
        -\psi^2+\left<z,\nabla_{z}\psi^2\right>, \psi^1-\left<z,\nabla_{z}\psi^1\right>\right)
    \end{align}
    is a normal vector field along $G$.
\end{lemma}

\begin{proof}
    We begin by noticing that $(\nabla_{z}\psi^1,-1,0)$ and $(\nabla_{z}\psi^2,0,-1)$ are normal to the graph at any $z\in V$, since we readily observe that $\tau_i(y)=(0,\ldots,1,0,\ldots,\p_i\psi^1(z),\p_i\psi^2(z))\in T_yG$ and $\left\{\tau_1,\ldots,\tau_n\right\}$ forms a basis of $TG$.
    Thus we can express any normal vector $\tilde N(y)$ by
    \begin{align}\label{eq-N(y)-linearCombination}
        \tilde N(y)=a_1(y)(\nabla_{z}\psi^1,-1,0)+a_2(y)(\nabla_{z}\psi^2,0,-1),
    \end{align}
    where $a_1,a_2$ are continuous on $G$.
    
    Since $G\subset\mfS^{n+1}$, we know that $(z,\psi^1(z),\psi^2(z))=y=\nu_{\mfS^{n+1}}(y)\perp \tilde N(y)$ and a direct computation gives
    \begin{align*}
        a_1(y)\left(\left<z,\nabla_{z}\psi^1\right>-\psi^1(x')\right)+a_2(y)\left(\left<z,\nabla_{z}\psi^2\right>-\psi^2(z)\right)=0.
    \end{align*}
    In view of this, we may choose
    \begin{align*}
    a_1(y)=-\left<z,\nabla_{z}\psi^2\right>+\psi^2(z),\quad a_2(y)=\left<z,\nabla_{z}\psi^1\right>-\psi^1(z),
    \end{align*}
    and it follows that \eqref{appendix-eq-1} is valid.
\end{proof}

\begin{lemma}\label{Lem-appendix-2}
    For the codimension-2 graph $G$ and $\tilde N$ as in \cref{Lem-appendix-1}, if at $z=\vec0$, $y=(\vec0,1,0)$ and $\tilde N(y)=(\vec0,0,1)$, then 
    \begin{align*}
        \psi^1(\vec0)=1, \psi^2(\vec0)=0,\quad \nabla_z\psi^1(\vec0)=\nabla_z\psi^2(\vec0)=\vec0.
    \end{align*}
    \begin{proof}
    Using the definition of $y$ and $\tilde N(y)$ to verify the condition: at $z=\vec0$, it holds that $y=(\vec0,1,0)$ and $\tilde N(y)=(\vec 0,0,1)$, one readily finds, \begin{align}\label{eq-LemAppendix2-1}
        \psi^1(\vec0)=1,\psi^2(\vec0)=0,\quad\nabla_z\psi^2(\vec0)=\vec0.
    \end{align}
    On the other hand, since $G\subset\mfS^{n+1}$, we have
    \begin{align*}
        \vert z\vert^2+\psi^1(z)^2+\psi^2(z)^2=1.
    \end{align*}
    Thanks to the $C^1$-differentiability of $\psi^i$, we can take directional derivative near $z=0$ along $e_1,\ldots,e_n$ to obtain
    \begin{align*}
        z_i+\psi^1(z)\p_i\psi^1(z)+\psi^2(z)\p_i\psi^2(z)=0,
    \end{align*}
    where $z_i$ denotes the $i$-th component of $z\in\mfR^n$. In particular, this, together with \eqref{eq-LemAppendix2-1}, shows that $\nabla_z\psi^1(\vec0)=\vec0$, which completes the proof.
\end{proof}
\end{lemma}
}


\printbibliography

\end{document}